\documentclass[12pt]{amsart}
\usepackage{amsmath,amssymb}

\usepackage[margin=2.0cm]{geometry}
\usepackage{hyperref}
\usepackage{multirow}

\usepackage{color}
\usepackage{enumerate}
\usepackage{tikz,lmodern}
\usepackage{amsfonts}
\usetikzlibrary{arrows, decorations.markings}

\usepackage{amsmath}
\usepackage{amsfonts}
\usepackage{amssymb}
\usepackage{mathrsfs}
\usepackage{caption}
\usepackage{latexsym}
\usepackage{color}
\usepackage{hyperref}
\usepackage{url}
\usepackage{enumerate}
\usepackage{arydshln} 
\usepackage{multicol} 
\usepackage{tikz}
\usepackage{booktabs}
\usetikzlibrary{decorations,arrows}
\usetikzlibrary{decorations.markings}

\renewcommand{\proof}{\noindent{\it Proof.\ \ }}
\renewcommand{\qed}{\ifmmode\square\else\nolinebreak\hfill
$\Box$\fi\par\vskip12pt}

\renewcommand\a{\alpha}  \renewcommand\b{\beta}  \renewcommand\d{\delta}
\def\s{\sigma} \def\g{\gamma}
\newcommand\e{\varepsilon} 
\def\th{\theta}

\newcommand\Ga{\mathrm{\Gamma}}

       \newcommand\D{\mathrm{D}}

\newcommand\Q{\mathrm{Q}}

\newcommand\mz{\mathbb{Z}}

          \newcommand\Aut{\mathrm{Aut}}
   \newcommand\Cay{\mathrm{Cay}}

\newcommand\Hol{\mathrm{Hol}}

\def\mod{\hbox{\rm \ mod}\,}

             \newcommand\Sym{\mathrm{Sym}}
          
            \newcommand\la{\langle}
\newcommand\ra{\rangle}

\newtheorem{theorem}{Theorem}[section]%
\newtheorem{lemma}[theorem]{Lemma}%
\newtheorem{corollary}[theorem]{Corollary}%
\newtheorem{proposition}[theorem]{Proposition}%

\title{Generalized quaternion NCI-groups, NNN-groups and NNND-groups}

\author[J.-F.~Yang]{Jun-Feng Yang}
\address{School of mathematics and statistics, Beijing Key Laboratory of Biological Big Data and Topological Statistics, Beijing Jiaotong University, Beijing, 100044, P.R. China}
\email{jf.yang@bjtu.edu.cn}

\author[J-L.~Du]{Jia-Li Du$^*$}
\address{School of Mathematical Sciences, Ministry of Education Key Laboratory of NSLSCS, Nanjing Normal University, Nanjing, 210023, P.R. China}
\email{dujl@njnu.edu.cn}

\author[Y.-Q.~Feng]{Yan-Quan Feng}
\address{School of mathematics and statistics, Beijing Key Laboratory of Biological Big Data and Topological Statistics, Beijing Jiaotong University, Beijing, 100044, P.R. China}
\email{yqfeng@bjtu.edu.cn}

\author[Y.S.~Kwon]{Young Soo Kwon}
\address{Mathematics, Yeungnam University, Kyongsan 712-749, Republic of Korea}
\email{ysookwon@ynu.ac.kr}

\thanks{$^*$ Corresponding author}
\thanks{The first and third authors were supported by the National Natural Science Foundation of China (12331013, 12271024). The second author was supported by the National Natural Science Foundation of China (12571370, 12271024) and National Key R\&D Program of China (211070B62501). The fourth author was
partially supported by the Basic Science Research Program through the National Research Foundation of
Korea (NRF) funded by the Ministry of Education (2018R1D1A1B05048450, 2021K2A9A2A11101586).}

\begin{document}
\date{}
\maketitle

\begin{abstract}
A Cayley (di)graph $\Cay(G,S)$ of a finite group $G$ is called CI if, for every Cayley (di)graph $\Cay(G,T)$ of $G$, $\Cay(G,S)\cong \Cay(G,T)$ implies that $S^{\sigma}=T$ for some $\sigma\in \Aut(G)$. The group $G$ is called an NDCI-group (resp. NCI-group) if every normal Cayley digraph (resp. graph) of $G$ is CI. It was shown that the generalized quaternion group $\Q_{4n}$ of order $4n$ ($n\geq 2$) is an NDCI-group if and only if either $n=2$ or $n$ is odd, but its NCI-group classification has been left as an open question. In this paper, we solve the question and prove that $\Q_{4n}$ is an NCI-group for every $n\geq 2$. A normal Cayley (di)graph of a group $G$ is called NNN if its automorphism group contains a non-normal regular subgroup isomorphic to $G$, and $G$ is called an NNND-group (resp. NNN-group) if it admits an NNN Cayley digraph (resp. graph). In this paper, we show that $\Q_{4n}$ is not an NNN-group for every $n\geq 2$, and is an NNND-group if and only if $n\geq 6$ and $n$ is even.

\bigskip
\noindent {\bf Key words:}  CI-digraph, NCI-group, NNN-group, NNND-group, generalized quaternion group.\\
{\bf 2020 Mathematics Subject Classification:} 20B25, 05C25.
\end{abstract}

\section{Introduction}

Let $\Ga$ be a digraph. Denote by $V(\Ga)$ and $A(\Ga)$ the vertex set and arc set of $\Ga$ respectively, and by $\Aut(\Ga)$ the automorphism group of $\Ga$. If $\Ga$ is a graph, $V(\Ga)$ and $E(\Ga)$ are the vertex set and edge set of $\Ga$, respectively. In this paper, we view a graph $\Ga$ as a digraph with the  same vertex set $V(\Ga)$ and arc set $\{(u,v)\ |\ \{u,v\}\in E(\Ga)\}$, and this causes no ambiguity. 

Let $G$ be a group and let $S$ be a subset of $G$ with the identity of $G$ not in $S$. The {\em Cayley digraph $\Cay(G,S)$} of $G$ with respect to $S$ is defined to have vertex set $G$ and arc set $\{(g,sg)\ |\ g\in G, s\in S\}$, while $\Cay(G,S)$ is called a {\em Cayley graph} when $S=S^{-1}$. Clearly, $\Ga$ is connected if and only if $\la S\ra=G$, that is, $S$ generates $G$. For $g\in G$, the right multiplication map $R(g): x\mapsto xg$, $x\in G$, is an automorphism of $\Cay(G,S)$. Write $R(G)=\{R(g)\mid g\in G\}$. Then $R(G)$ is a regular subgroup of $\Aut(\Ga)$. It was shown that a digraph is isomorphic to a Cayley digraph of a group $G$ if and only if its automorphism group contains a regular subgroup isomorphic to $G$; see \cite{DMM}. A basic relation between $R(G)$ and $\Aut(\Cay(G,S))$ was given by Godsil~\cite{Godsil} and he proved that the normalizer of $R(G)$ in $\Aut(\Cay(G,S))$ is precisely $R(G) \rtimes \Aut(G,S)$, where $\Aut(G,S)=\{\a\in \Aut(G)\mid S^{\a}=S\}$. A Cayley digraph $\Cay(G,S)$ of $G$ is called {\em normal} if $R(G)$ is a normal subgroup in  $\Aut(\Cay(G,S))$; see Xu~\cite{XuM}. The above Godsil's result implies that if $\Aut(\Cay(G,S))$ is normal, then  $\Aut(\Cay(G,S))=R(G) \rtimes \Aut(G,S)$.

A classical problem in algebraic graph theory is the Cayley isomorphism (CI) problem. A Cayley digraph $\Cay(G,S)$ is called {\em CI} if for any Cayley digraph $\Cay(G,T)$, whenever $\Cay(G,S)\cong \Cay(G,T)$, we have $S^{\s}=T$ for some $\s\in \Aut(G)$, and a group $G$ is called a {\em DCI-group} (resp. {\em CI-group}), if every Cayley digraph (resp. Cayley graph) of $G$ is CI. The study of CI originated from a conjecture proposed by \'Ad\'am~\cite{Adam}: every finite cyclic group is a CI-group. Though it was disproved by Elspas and Turner~\cite{Elspas}, the classifications of cyclic DCI-groups and CI-groups lasted for 30 years. It was known that the \'Ad\'am's conjecture is true for the cyclic group of certain order $n$: $n=p$~\cite{Djokovic,Elspas,Turner}; $n=2p$~\cite{Babai}; $n=pq$~\cite{Alspach,Godsil}; $n=4p$ with $p>2$~\cite{Godsil}; $(n, \phi(n))=1$~\cite{Palfy}, where  $p$ and $q$ are distinct primes and $\phi$ is the Euler's function. Finally, Muzychuk~\cite{Muzychuk,M.Muzychuk} proved that a cyclic group of order $n$ is a DCI-group if and only if $n=mk$ where $m=1,2,4$ and $k$ is odd square-free, and is a CI-group if and only if either $n=8,9,18$, or $n=mk$ where $m=1,2,4$ and $k$ is odd square-free, which is one of the remarkable achievements in the study of DCI-groups and CI-groups. For other groups, there are also a lot of works on this line, and we do not intend to give a full account on their study in general. Interested readers are referred to \cite{Kov,KR,Li,LLP,SM}, and in particular to \cite{AN,CL,D2,Feng,HM,J.M,Mu,N,So,Sp,Sp1} for elementary abelian groups. Finite DCI-groups and CI-groups  have been reduced to some special restricted groups; see \cite{Li} for DCI-groups and \cite{DMS,LLP} for CI-groups. However, it is very difficult to determine whether a particular group in the special restricted groups is a DCI-group or a CI-group. In fact, it is still an open question to classify dihedral or generalized quaternion CI-groups or DCI-groups. 

In recent years, CI have been considered for normal Cayley digraphs. Note that Li~\cite{Li} provided examples of normal Cayley digraphs that are not CI, and proposed the following question: Characterize normal Cayley digraphs which are not CI-digraphs. More recently, Ryabov~\cite{Ryabov2025} studied the CI-property of normal Cayley digraphs over abelian groups, and reduced the abelian case to abelian $p$-groups.
A group $G$ is called an {\em NDCI-group} if every normal Cayley digraph of $G$ is a CI-digraph, and an {\em NCI-group} if every normal Cayley graph of $G$ is a CI-graph. Note that every DCI-group is an NDCI-group, and every CI-group is an NCI-group, but the converses are not true. Xie, Feng, Ryabov and Liu~\cite{XFRL} proved that a cyclic group of order $n$ is NDCI if and only if $8\nmid n$, and is NCI if and only if either $n=8$ or $8\nmid n$. For dihedral groups, Xie, Feng and Zhou~\cite{XFZ} proved that a dihedral group of order $2n$ is NDCI or NCI if and only if $n=2,4$, or $n$ is odd. In 2022, Xie, Feng and Kwon~\cite{XFK} proved that a generalized quaternion group of order $4n$ is NDCI if and only if either $n=2$ or $n$ is odd, but left the NCI case as an open question. In this paper, we solve this problem; see Theorem~\ref{mainth1}.

A Cayley digraph $\Gamma=\Cay(G,S)$ is called {\em NNN} when $\Gamma$ is normal and $\Aut(\Gamma)$ has a non-normal regular subgroup isomorphic to $G$. A group $G$ is called an {\em NNND-group} or {\em NNN-group} if there exists an NNN Cayley digraph or graph of $G$, respectively. Since a graph is viewed as a digraph, every NNN-group is an NNND-group, but the converse is not true. 

The first NNN-graph was constructed by Royle~\cite{Royle} in 2008, which is a normal Cayley graph of $\mz_2^6$. In 2010, Giudici and Smith~\cite{GS} built another NNN-graph, a strongly regular Cayley graph of $\mz_6^2$. In 2011, Bamberg and Giudici~\cite{BM} found the first infinite family of NNN-graphs as point graphs of a certain family of generalized quadrangles. Moreover, Xu~\cite{XuY17} gave further infinite families of NNN-graphs using Cartesian product, direct product, and strong product in 2017.

In 2018, Xu~\cite{XuY18} proved that every elementary abelian $2$-group $\mz_2^d$ with $d\ge 6$ is an NNN-group, and that every non-abelian simple group is not an NNND-group. In 2022, Xu~\cite{XuY22} proved that the symmetric group $S_n$ of degree $n$ is an NNN-group if and only if $n\ge 5$. In 2019, Giudici, Morgan, and Xu~\cite{GLX} showed that every cyclic group is not an NNN-group. Using a different method from \cite{GLX}, Yang et al.~\cite{Yang} also proved that there is no cyclic NNN-group and NNND-group, and further proved that a dihedral group $\D_{2n}$ of order $2n$ is an NNN-group or NNND-group if and only if $n\ge 6$ is even and $n\neq 8$. In this paper, we classify generalized quaternion NNN-groups and NNND-groups. As a result, we find infinitely many generalized quaternion NNND-groups but not NNN-groups, and to the best of our knowledge, this is the first discovery of an infinite family of groups with such property.   

Let $n\geq 2$ and let $\Q_{4n}$ be the generalized quaternion group of order $4n$. Our first result in this paper is about NCI-group. 
 
\begin{theorem}\rm\label{mainth1}
Every generalized quaternion group $\Q_{4n}$ $(n\geq 2)$ is an NCI-group.
\end{theorem}

By Babai~\cite{Babai}, a Cayley digraph $\Ga=\Cay(G, S)$ of a group $G$ is CI if and only
if all regular subgroups of $\Aut(\Ga)$ isomorphic to $G$ are conjugate. From this we know that if $\Ga$ is
a normal Cayley digraph of $G$, then $\Ga$ is CI if and only if $R(G)$ is the unique regular subgroup
of $\Aut(\Ga)$ isomorphic to $G$. This implies that every NDCI-group is not an NNND-group and every NCI-group is not an NNN-group. Thus, we have the following corollary.  

\begin{corollary}\rm\label{corollary2}
Every generalized quaternion group $\Q_{4n}$ $(n\geq 2)$ is not an NNN-group.
\end{corollary}

For generalized quaternion NNND-group, we have the following result.

\begin{theorem}\rm\label{mainth3}
A generalized quaternion group $\Q_{4n}$  $(n\geq 2)$  is an NNND-group if and only if $n\ge 6$ and $n$ is even.
\end{theorem}

The paper is organized as follows. In Section 2, we present some definitions and preliminary results for later use. Section 3 proves Theorem~\ref{mainth1}, and Section 4 proves Theorem~\ref{mainth3}. In what follows, all digraphs considered are finite and simple, and all groups are finite. 

\section{Preliminaries}

Let $\Ga=\Cay(G,S)$ be a Cayley digraph of a group $G$. Recall that $\Aut(G,S)\leq \Aut(\Ga)_1$, where $\Aut(\Ga)_1$ is the stabilizer of the identity vertex $1$ in $\Aut(\Ga)$. The following proposition provides a criterion for $\Ga$ to be a normal Cayley digraph.

\begin{proposition}\rm\label{normal}(\cite[Proposition 1.5]{XuM}) Let $\Ga=\Cay(G,S)$ be a Cayley digraph of a group $G$ with respect to $S$. Then $\Ga$ is normal if and only if $\Aut(\Ga)=R(G)\rtimes \Aut(G,S)$, and if and only if $\Aut(\Ga)_1=\Aut(G,S)$. 
\end{proposition}

The following is a classification of generalized quaternion NDCI-groups.

\begin{proposition}\rm\label{ndci}(\cite[Theorem 1.1]{XFK})
Let $n\ge 2$ be a positive integer and let $\Q_{4n}$ be the generalized quaternion group of order $4n$. Then $\Q_{4n}$ is an NDCI-group if and only if either $n=2$ or $n$ is odd.
\end{proposition}


Let $G$ be a group and let $L\subseteq \Aut(G)$.
Set $F_G(L)=\{g\  |\ g^l=g \  {\rm for \ all}\  l\in L\}$, the set of fixed-points of $L$ in $G$.
The following proposition provides some sufficient conditions for a Cayley digraph to be non-normal.

\begin{proposition}\rm\label{MainPro} (\cite[Theorem 2.2]{XFK})
  Let $\Cay(G,S)$ be a Cayley digraph. Let $1\neq L\le \Aut(G,S)$ and let $K\unlhd G$ such that for every right coset $Kg$ in $G$, either $L$ fixes $Kg$ pointwise, or $Kg$ is an orbit of $L$. Assume that one of the following holds:
  \begin{itemize}
    \item[(1)] $|G:F_G(L)|>2$;
    \item[(2)] $|G:F_G(L)|=2$, and there are $g\in G\setminus F_G(L)$ and $k\in K$ such that $k^g\neq k^{-1}$;
    \item[(3)] $|G:F_G(L)|=2$, and there is $1\neq \gamma\in \Aut(G,S)$ such that $F_G(\la \gamma\ra)\neq F_G(L)$ and $\gamma$ fixes every coset of $K$ in $G$ setwise.
  \end{itemize}
  Then $\Cay(G,S)$ is non-normal.
\end{proposition}

Let $G$ be a group and $x,y\in G$. Denote by $[x,y]$ the abbreviation for the {\em commutator} $x^{-1}y^{-1}xy$ of $x$ and $y$, and by $[H, K]$ the subgroup generated by all commutators $[x, y]$ with $x \in H$ and $y \in K$, where $H$ and $K$ are subsets of $G$. Clearly, $[H, K]=1$ if and only if $H$ and $K$ commute pointwise. The following  proposition is a basic property of commutators and its proof is straightforward.

\begin{proposition}\rm\label{commutator}
Let $G$ be a group. Then, for any $x, y, z \in G$, $[xy, z]=[x, z]^y[y, z]$ and $[x, yz]=[x, z][x, y]^z$.
\end{proposition}

The {\em commutator} subgroup $G'$ of a group $G$ is the subgroup generated by all commutators $[x, y]$ for any $x, y \in G$. With Proposition~\ref{commutator}, it is easy to prove that if $G$ is generated by a subset $M$, then $G'$ is generated by all conjugates of elements $[x_i, x_j]$ with $x_i, x_j \in M$  in $G$.

\begin{proposition}\rm\label{Derived} (\cite[Hilfsatz \uppercase\expandafter{\romannumeral3}.1.11]{Huppert})
Let $G$ be a group, $S \subseteq G$ and $G = \la S \ra$. Then $G' = \langle [x_i, x_j]^g\ |\ x_i, x_j \in S, g \in G\rangle$.
\end{proposition}

Let $p$ be a prime. A $p$-group $G$ is called a {\it regular $p$-group} if for any two elements $x$ and $y$ in $G$, there exist $d_1, d_2,\dots,d_r$ in $\la x, y\ra'$ such that $(xy)^p = x^py^p d_1^pd_2^p\dots d_r^p$. The following proposition gives a sufficient condition for a $p$-group to be regular.

\begin{proposition}\rm\label{pgroup} (\cite[\uppercase\expandafter{\romannumeral3}, Theorem 10.2]{Huppert})
Let $G$ be a $p$-group with $p$ an odd prime. If $G'$ is cyclic then $G$ is regular.

\end{proposition}



\section{Proof of Theorem~\ref{mainth1}\label{S3}}

For a group $G$, write $|G|$ for the order of $G$, and set $o(g)=|\la g\ra|$ for   $g\in G$, called the order of $g$.
For a positive integer $n$, let
\begin{equation}\label{factorizationof2n}
 2n=\prod_{i=1}^{s+1}p_i^{k_i} \mbox{ with } p_1>p_2>\cdots>p_s>p_{s+1}=2,
\end{equation}
the distinct prime factorization of $2n$, where $k_i\geq 1$ is a positive integer for every $1\leq i\leq s+1$ and $p_i$ is an odd prime for every $1\le i\le s$. Note that if $s=0$ then $n$ is a power of $2$. Let $C_{2n}=\la a\ra$ be the cyclic group of order $2n$.
Then we may write
\begin{equation}\label{C_2nfactorization}
  C_{2n}=C_{p_1^{k_1}}\times\cdots\times C_{p_s^{k_s}}\times C_{p_{s+1}^{k_{s+1}}}= \la a_1\ra \times\cdots \times \la a_s\ra\times \la a_{s+1}\ra \mbox{ and } a=a_1\cdots a_sa_{s+1},
\end{equation}
where $o(a_i)=p_i^{k_i}$ for $1\le i\le s+1$. For any $x\in C_{2n}$, we have the decomposition:
\begin{equation}\label{decomposition}
  x=x_1x_2\cdots x_sx_{s+1}, \mbox{ where }x_i\in C_{p_i^{k_i}} \mbox{ for }i=1,2,\cdots,s+1.
\end{equation}
Note that $x_i$ is uniquely determined by $x$ and $o(x_i)$ is a $p_i$-power. We call $x_i$ the {\em $i$-part} of $x$. 
 
Let $n\geq 2$ and let $\Q_{4n}$ be the generalized quaternion group of order $4n$, that is,
\begin{equation}\label{Q4n}
\Q_{4n}=\la a,b\ |\ a^{2n}=1, b^2=a^n, b^{-1}ab=a^{-1}\ra.
\end{equation}
Now $C_{2n}\leq \Q_{4n}$ and $\Q_{4n}=C_{2n}\cup bC_{2n}$.
For $x\in C_{2n}$ and  $\th\in \Aut(C_{2n})$, let $\s_x, \th\in \Aut(\Q_{4n})$ be induced by
\begin{equation}\label{s_ai}
  \s_{x}: a \mapsto a, b\mapsto xb, \mbox{ and } \th: a\mapsto a^\th, b\mapsto b, \mbox{ respectively}.
\end{equation}
Then $\la \s_a\ra\cong C_{2n}$ and $o(\s_{a_i})=o(a_i)=p_i^{k_i}$.  Let $n\geq 3$.
From \cite[Lemma 3.1]{XFK} we have
\begin{equation}\label{autoq_4n}
 \Aut(\Q_{4n})=\la \s_a\ra \rtimes \Aut(C_{2n}),
\end{equation}
where $\la \s_a\ra$ is the kernel of $\Aut(\Q_{4n})$ acting on $C_{2n}$ that is characteristic in $\Q_{4n}$, and $\s_{x}^{\th}=\s_{x^{\th}}$ for all $x\in C_{2n}$ and $\th\in \Aut(C_{2n})$.
From \cite[Theorem 4.7]{Rot} we have
\begin{equation}\label{autoc_2nfactorization}
\Aut(C_{2n})=\Aut(C_{p_1^{k_1}})\times \cdots \times \Aut(C_{p_s^{k_s}})\times \Aut(C_{p_{s+1}^{k_{s+1}}}),
\end{equation}
where $\Aut(C_{p_i^{k_i}})$ is viewed as a subgroup of $\Aut(C_{2n})$ by defining for $\d_i\in \Aut(C_{p_i^{k_i}})$ that $a_j^{\d_i}=a_j$ for all $j\neq i$. Thus,  $\Aut(C_{p_i^{k_i}})$ is a subgroup of $\Aut(\Q_{4n})$.



Let $k_i\geq 2$ for some $1\le i\le s+1$. Let $\a_i\in \Aut(C_{p_i^{k_i}})\leq \Aut(\Q_{4n})$, induced by
\begin{equation}\label{auorderp_i}
\a_i: a_i\mapsto a_i^{p_i^{k_i-1}+1}, a_j\mapsto a_j \mbox{ and } b\mapsto b \mbox{ for every } j\neq i.
\end{equation}
Then $\a_i$ has order $p_i$. Let $$\b_i=\s_{a_i}^{p_i^{k_i-1}}.$$ Then $\b_i$ also has order $p_i$  for every $1\le i\le s+1$, which is induced by
\begin{equation}\label{auorderp_i2}
\b_i: a\mapsto a \mbox{ and } b\mapsto a_i^{p_i^{k_i-1}}b.
\end{equation}

Note that $k_j\geq 1$ for all $1\le j\le s+1$. Let $\e_j\in \Aut(C_{p_j^{k_j}})\leq \Aut(\Q_{4n})$ be induced by
\begin{equation}\label{involutions}
    \e_j: a_j\mapsto a_j^{-1}, b\mapsto b, \mbox{ and } a_i\mapsto a_i \mbox{ for }  \mbox{ every } 1\le i\le s+1 \mbox { with } i\not=j.
\end{equation}
Then $\e_j$ is an involution for every $1\leq j\leq s$. Furthermore, $\e_{s+1}=1$ for $k_{s+1}=1$, and  $\e_{s+1}$ is an involution for $k_{s+1}\geq 2$. Set \begin{equation}\label{epson}
    \e=\e_1\e_2\cdots\e_{s+1}.
\end{equation}
Then $x^\e=x^{-1}$ and $(xb)^\e=x^{-1}b$ for any $x\in C_{2n}$.

Let $k_{s+1}\ge 3$. Recall that $p_{s+1}=2$. Then $\Aut(C_{p_{s+1}^{k_{s+1}}})\cong C_2\times C_{2^{k_{s+1}-2}}$, and $\Aut(C_{p_{s+1}^{k_{s+1}}})$ has exactly three involutions, which are also  automorphisms of $\Q_{4n}$. These involutions are $\a_{s+1}$, $\e_{s+1}$ and $\a_{s+1}\e_{s+1}$.



\medskip 
For a digraph (resp. a graph) $\Ga$  and two subsets $X$ and $Y$ of $V(\Ga)$, denote by $[X]_{\Ga}$ the induced subdigraph (resp. a subgraph) of $X$ in $\Ga$, and if there is no confusion, we delete the subscript $\Ga$, that is, $[X]_{\Ga}=[X]$. Denote by $[X,Y]$ the digraph (resp. graph) with vertex set $X\cup Y$ and arc set $\{(x,y)\ |\ x\in X, y\in Y, (x,y)\in A(\Ga)\}$ (resp. edge set $\{\{x,y\}\ |\ x\in X, y\in Y, \{x,y\}\in E(\Ga)\}$). When $X$ and $Y$ are disjoint, $[X,Y]$ is a bipartite graph if $\Ga$ is a graph, called {\em the induced bipartite subgraph between $X$ and $Y$}, and is a bipartite digraph with arcs from $X$ to $Y$ if $\Ga$ is a digraph, called {\em the induced bipartite subdigraph from $X$ to $Y$}. For a group $G$ and a subgroup $H$ of $G$, denote by $[G:H]$ the set of right cosets of $H$ in $G$. The following lemma provides some sufficient conditions for a Cayley digraph of $\Q_{4n}$ to be non-normal. 

\medskip
\begin{lemma}\rm\label{lm1}
 Let $n\geq 2$ and let $\Ga=\Cay(\Q_{4n},S)$ be a Cayley digraph of the generalized quaternion group $\Q_{4n}$. Then we have the following:
 \begin{itemize}
    \item[(1)] If $\langle \b_i \rangle\a_i\cap \Aut(\Q_{4n},S)\not=\emptyset$ for some $1\le i\le s$, then $\Ga$ is non-normal;
    \item[(2)] Let $1\not=\a\in \Aut(C_{2n})\cap \Aut(\Q_{4n},S)$ and $1\not=\s_{a^r}\in \Aut(\Q_{4n},S)$. If $\a$ fixes $[C_{2n}:\la a^r\ra]$ pointwise, then $\Ga$ is non-normal. In particular, if $\Ga$ is a graph and $\la\s_a\ra\a_{s+1}\cap \Aut(\Q_{4n},S)\not=\emptyset$, then $\Ga$ is non-normal.
  \item[(3)] Let $\Ga$ be a graph and $n>2$. If $\la\s_a\ra\e\a_{s+1}\cap \Aut(\Q_{4n},S)\not=\emptyset$, then $\Ga$ is non-normal.
  \end{itemize}
\end{lemma}

\proof To prove (1), let $\langle \b_i \rangle\a_i\cap \Aut(\Q_{4n},S)\not=\emptyset$ for some $1\le i\le s$. Since both $\a_i$ and $\b_i$ have order $p_i$, we have $k_i\geq 2$ and $\b_i^r\a_i\in \Aut(\Q_{4n},S)$ for some $0\leq r< p_i$. Write  $\Lambda=\la \b_i^r\a_i \ra\leq \Aut(\Q_{4n},S)$.
Recall that $\b_i=\s_{a_i}^{p_i^{k_i-1}}$. Then $\langle\b_i\rangle$ is the unique subgroup of order $p_i$ in $\langle\s_a\rangle$ and hence characteristic in $\langle \s_a\rangle$. This, together with Eq~(\ref{autoq_4n}), implies that $\langle \b_i\rangle\unlhd \Aut(\Q_{4n})$, and hence $\langle \b_i\rangle \langle \a_i\rangle\leq \Aut(\Q_{4n})$. Thus, $\langle \b_i\rangle \langle \a_i\rangle$ is an elementary abelian $p_i$-group of order $p_i^2$. It follows that $|\Lambda|=p_i$.

Since $k_i\geq 2$, $C_{2n}=\langle a\rangle$ is the unique subgroup of order $2n$ in $\Q_{4n}$ and so characteristic in $\Q_{4n}$. Set $H=\la a_1,\cdots, a_{i-1},a_{i+1},\cdots ,a_s,a_{s+1}\ra$, $K=\la a_i^{p_i^{k_i-1}}\ra$ and $L=\la a_1,\cdots, a_{i-1},a_i^{p_i},a_{i+1},\cdots ,a_s,a_{s+1}\ra$. Since $C_{2n}$ is cyclic, each of $K$, $H$ and $L$ is characteristic in $\Q_{4n}$, and hence normal in $\Q_{4n}$. Since  $k_i\geq 2$, we have $K\leq L$. Furthermore, $|H|=2n/p_i^{k_i}$, $|K|=p_i$ and $|L|=2n/p_i$.

Let $F_{\Q_{4n}}(\Lambda)$ be the set of fixed-points of $\Lambda$ in $\Q_{4n}$. From  Eqs~(\ref{auorderp_i}) and (\ref{auorderp_i2}) it is easy to see that $L=\la a_1,\cdots,a_i^{p_i},\cdots,a_s,a_{s+1}\ra\leq F_{\Q_{4n}}(\Lambda)$.  Furthermore,
$(a_i^{-r}b)^{\b_i^r\a_i}=(a_i^{-r}a_i^{rp_i^{k_i-1}}b)^{\a_i}=
a_i^{-r(p_i^{k_i-1}+1)}a_i^{rp_i^{k_i-1}}b=a_i^{-r}b$, that is, $a_i^{-r}b\in F_{\Q_{4n}}(\Lambda)$. Since $L\unlhd \Q_{4n}$ and $(a_i^{-r}b)^2=b^2=a^n\in L$, we have $L\langle a_i^{-r}b \rangle\leq \Q_{4n}$ and $|L\langle a_i^{-r}b\rangle|=2|L|=4n/p_i$. Then $|\Q_{4n}:L\langle a_i^{-r}b \rangle|=p_i$, and since $p_i$ is a prime and $F_{\Q_{4n}}(\Lambda)\not=\Q_{4n}$, we have $F_{\Q_{4n}}(\Lambda)=L\langle a_i^{-r}b \rangle$.

Note that $|\Q_{4n}:F_{\Q_{4n}}(\Lambda)|=p_i>2$. Let $T=L\langle a_i^{-r}b \rangle=F_{\Q_{4n}}(\Lambda)$. Since $a_i\not\in T$, we have  $\Q_{4n}=\bigcup_{j=0}^{p_i-1}Ta_i^j=\bigcup_{x\in T}\bigcup_{j=0}^{p_i-1}Ka_i^jx$. Noting that $a_i^{\b_i^r\a_i}=a_i^{\a_i}=a_i^{p_i^{k_i-1}}a_i\in Ka_i$, $\Lambda$ fixes every coset $Ka_i^jx$ setwise, and since $|\Lambda|=p_i=|Ka_i^jx|$, $\Lambda$ is either transitive on $Ka_i^jx$, or fixes $Ka_i^jx$ pointwise. By Proposition~\ref{MainPro}(1), $\Ga$ is non-normal.

\medskip
To prove (2), assume that $\a$ fixes $[C_{2n}:\la a^r\ra]$ pointwise. Denote by $[C_{2n}b:\la a^r\ra]$ the set of right cosets of $\la a^r\ra$ in $C_{2n}b$. By Eq~(\ref{s_ai}), $[C_{2n}b:\la a^r\ra]$ is the orbit-set of $\la \s_{a^r}\ra$ on $C_{2n}b$, and since $\s_{a^r}\in \Aut(\Q_{4n},S)\leq \Aut(\Ga)$, for every coset $\la a^r\ra y$ with $y\in C_{2n}b$ we have that either $1$ has no out-neighbor in $\la a^r\ra y$, or $1$ is adjacent to every vertex in $\la a^r\ra y$. Let $m=|\la a^r\ra|\not=1$. Since $\la R(a^r)\ra\leq \Aut(\Ga)$, the bipartite induced subdigraph $[\la a^r\ra, \la a^r\ra y]$ has no out-neighbor or is a complete bipartite digraph with arcs from $\la a^r\ra$ to $\la a^r\ra y$, that is, $[\la a^r\ra, \la a^r\ra y]=2mK_1$ or $\vec{K}_{m,m}$. Similarly, $[\la a^r\ra y,\la a^r\ra]=2mK_1$ or $\vec{K}_{m,m}$. Since $\la R(a)\ra\leq \Aut(\Ga)$, for all $x\in C_{2n}$ and $y\in C_{2n}b$, we have $[\la a^r\ra x, \la a^r\ra y]=2mK_1$ or $\vec{K}_{m,m}$, and $[\la a^r\ra y,\la a^r\ra x]=2mK_1$ or $\vec{K}_{m,m}$. This implies that any automorphism of the induced subdigraph $[C_{2n}]$ of $\Gamma$ fixing $[C_{2n}:\la a^r\ra]$ pointwise, can be extended to an automorphism of $\Ga$ by fixing $C_{2n}b$ pointwise. Recall that $\a$ fixes $[C_{2n}:\la a^r\ra]$ pointwise. Since $\a\in \Aut(C_{2n})\cap \Aut(\Q_{4n},S)\leq \Aut(\Ga)$, the restriction of $\a$ on $C_{2n}$ is an automorphism of the induced subdigraph $[C_{2n}]$ of $\Gamma$ fixing $[C_{2n}:\la a^r\ra]$ pointwise, the restriction can be extended to an automorphism, say $\a'$, of $\Ga$ such that $x^{\a'}=x^\a$ for $x\in C_{2n}$ and $y^{\a'}=y$ for $y\in C_{2n}b$. Since $1\not=\a\in \Aut(C_{2n})\cap \Aut(\Q_{4n},S)$, $\a$ has a non-trivial action on $C_{2n}$, implying $\a'\not=1$.

Suppose $\Ga$ is normal. By Proposition~\ref{normal}, $\a'\in A_1=\Aut(G,S)$, and since $\a'$ fixes  $C_{2n}b$ pointwise, it fixes $\Q_{4n}=\la C_{2n}b\ra$ pointwise. It follows that $\a'=1$, a contradiction. Thus, $\Ga$ is non-normal.

Now assume that $\Ga$ is a graph and $\la\s_a\ra\a_{s+1}\cap \Aut(\Q_{4n},S)\not=\emptyset$. Then $S=S^{-1}$. Clearly, $(a^ib)^{\s_{a^n}}=a^{n+i}b=(a^ib)^{-1}$ for any $a^ib\in C_{2n}b$ and $\s_{a^n}$ fixes $C_{2n}$ pointwise. Thus, $S^{\s_{a^n}}=S$, that is, $\s_{a^n}\in \Aut(\Q_{4n},S)$. Note that $a^n=a_{s+1}^{2^{k_{s+1}-1}}$ because both are the unique involution in $\la a\ra$. By Eqs~(\ref{auorderp_i}) and (\ref{s_ai}), $\la\s_a\ra\a_{s+1}\cap \Aut(\Q_{4n},S)$ fixes $[C_{2n}:\la a^n\ra]$ pointwise, and since $\la\s_a\ra\a_{s+1}\cap \Aut(\Q_{4n},S)\not=\emptyset$, the above proof implies that $\Ga$ is non-normal.  
\medskip

To prove (3), let $\gamma\in\la \s_a\ra \e\a_{s+1}\cap \Aut(\Q_{4n},S)\neq \emptyset$. Since $\a_{s+1}$ has order $2$, we have $k_{s+1}\geq 2$. By Eqs~(\ref{involutions}) and (\ref{epson}), $a_i^{\g}=a_i^{\e\a_{s+1}}=a_i^{\e_i}=a_i^{-1}$ for $1\le i\le s$, and $a_{s+1}^{\g}=a_{s+1}^{\e\a_{s+1}}=a_{s+1}^{\e_{s+1}\a_{s+1}}=a_{s+1}^{2^{k_{s+1}-1}-1}=a_{s+1}^{-1}a^n$. 
For $x\in\la a\ra$, by Eq~(\ref{decomposition}) we may write $x=x_1\cdots x_sx_{s+1}$, where $x_i$ is the $i$-part of $x$ and $x_i\in\la a_i\ra$. Then we have 
$$x^{\gamma}=x^{-1} \mbox{ if } o(x_{s+1})\neq 2^{k_{s+1}}, \mbox{ and } x^{\gamma}=x^{-1}a^n \mbox{ if } o(x_{s+1})= 2^{k_{s+1}}.$$ 
It follows that $x^\g=x^{-1}$ for any $x\in \la a^2\ra$ and $y^\g=y^{-1}a^n$ for any $y\in \la a^2\ra a$.  

Let $M=\la a^n\ra$, the unique subgroup of order $2$ in $C_{2n}$. Then $M\leq \la a^2\ra$ as $k_{s+1}\geq 2$. Take any $x\in \la a^2\ra$ and $y\in \la a^2\ra a$. We claim that $[Mx, My]$ is either the empty graph $4K_1$ of order $4$ or the complete bipartite graph $K_{2,2}$ of order $4$.  

Take $z\in \la a^2\ra a$. Then $o(z)=2n$. Assume that $[M, Mz]$ has an edge. Note that $R(M)\leq \Aut(\Ga)$ interchanges the two vertices in every coset of $M$ in $C_{2n}$. Thus, we may assume that $\{1,z\}$ is an edge, that is, $z\in S$. Since $S=S^{-1}$, we have $z^{-1}\in S$. Note that $Mz\subseteq \la a^2\ra a$ and  $Mz^{-1}\subseteq \la a^2\ra a$. Since $o(z)=2n$ and $n>2$, we have $Mz\not=Mz^{-1}$. Furthermore, $z^{\g}=z^{-1}a^n\in Mz^{-1}$, and $\{1,z^{-1}a^n\}$ is an edge, that is $z^{-1}a^n\in S$. Thus, $za^n=(z^{-1}a^n)^{-1}\in S$, and hence $1$ is adjacent to every vertex in $Mz$. Since $R(M)\leq \Aut(\Ga)$, we have $[M,Mz]=K_{2,2}$. It follows that $[M, Mz]=4K_1$ or $K_{2,2}$. Since $R(\la a^2\ra)\leq \Aut(\Ga)$, the arbitrariness of $z$ implies that $[Mx,My]=4K_1$ or $K_{2,2}$ for any $x\in \la a^2\ra$ and $y\in \la a^2\ra a$, as claimed. 

Since $\Ga$ is a graph, $\s_{a^n}\in \Aut(\Q_{4n},S)$, and hence $[My,Mw]=4K_1$ or $K_{2,2}$ for any $y\in \la a^2\ra a$ and $w\in  C_{2n}b$. This, together with the claim above, implies that every automorphism of the induced subgraph $[\la a^2\ra a]$ fixing every coset of $M$ in   
$\la a^2\ra a$ can be extended to an automorphism of $\Ga$ by fixing every vertex in $\la a^2\ra\cup C_{2n}b$. In particular, $R(a^n)$ fixes every coset of $M$ in $\Q_{4n}$, and its restriction on $\la a^2\ra a$ can be extended to an automorphism $\rho$ of $\Ga$ such that $x^\rho=xa^n$ for $x\in \la a^2\ra a$ and $x^\rho=x$ for $x\not\in \la a^2\ra a$. Clearly, $\rho\not=1$.   

Suppose that $\Ga$ is normal. Then $\rho\in A_1=\Aut(\Q_{4n},S)$ by Proposition~\ref{normal}. Since $\rho$ fixes $\Q_{4n}\setminus \la a^2\ra a$ pointwise, it fixes $\Q_{4n}=\la C_{2n}b\ra$ pointwise, implying  $\rho=1$, a contradiction. Thus, $\Ga$ is non-normal. This completes the proof.\qed

\medskip

For a group $G$ and a prime $p$, denote by $G_p$ a Sylow $p$-subgroup of $G$. Recall that $R(G)=\{R(g)\mid g\in G\}$, where $R(g)$ is the right multiplication map of $g$ on $G$. Then $R(G)$ is a regular permutation subgroup of the symmetric group $\Sym(G)$ on $G$, and the automorphism group $\Aut(G)$ of $G$ is also a subgroup of $\Sym(G)$. The normalizer of $R(G)$ in the symmetric group $\Sym(G)$ is called the {\em holomorph} of $G$, denoted by $\Hol(G)$. By \cite[Lemma 7.16]{Rot}, $\Hol(G)=R(G)\rtimes\Aut(G)$. Thus, for any $T\leq \Aut(G)$, we have the semidirect product $R(G)\rtimes T$. It is easy to see that 
\begin{equation}\label{add}
R(g)^\a=R(g^\a) \mbox{ for all } g\in G \mbox{ and } \a\in \Aut(G).
\end{equation}

Let $n\geq 3$. By Eq~(\ref{autoq_4n}) and \cite[Lemma~3.1]{XFK}, we have 
\begin{equation}\label{holQ_4n1} \Hol(\Q_{4n})=R(\Q_{4n})\rtimes\Aut(\Q_{4n})=R(\Q_{4n})\rtimes(\la \s_a\ra \rtimes \Aut(C_{2n})).
\end{equation}

Recall that for all $1\le i,j\le s+1$ with $i\neq j$, $\Aut(C_{p_i^{k_i}})$ acts trivially on $C_{p_j^{k_j}}$.
Thus,
\begin{equation}\label{commute}
    [\Aut(C_{p_i^{k_i}}),\Aut(C_{p_j^{k_j}})]=1 \mbox{ and } [\Aut(C_{p_i^{k_i}}), \la \s_{a_j}\ra]=1, \mbox{ for all } i\neq j,
\end{equation}
that is, $\Aut(C_{p_i^{k_i}})$ commutes with $\Aut(C_{p_j^{k_j}})$ and $ \la \s_{a_j}\ra$ pointwise.
By Eq~(\ref{autoc_2nfactorization}), we have
\begin{equation}\label{holQ_4n2}
\Aut(\Q_{4n})=(\la \s_{a_1}\ra \rtimes \Aut(C_{p_1^{k_1}}))\times \cdots \times (\la \s_{a_s}\ra \rtimes \Aut(C_{p_s^{k_s}}))\times (\la \s_{a_{s+1}}\ra\rtimes \Aut(C_{p_{s+1}^{k_{s+1}}})).
\end{equation}

Let $c\in C_{2n}$ and $z\in C_{2n}b$. Then $\s_c^{-1}R(z)\s_c=R(z)^{\s_c}=R(z^{\s_c})=R(cz)=R(c)R(z)$, that is, $\s_cR(z)=R(c^{-1}z)\s_c$.
It follows that
\begin{equation}\label{s^R}
\s_c^{R(z)}=R(c)\s_c, \mbox{ for all } c\in C_{2n} \mbox{ and } z\in C_{2n}b.
\end{equation}

\medskip

In the remainder of this section, we use the notions and formulas in Eqs~(\ref{factorizationof2n})-(\ref{s^R}). 

\medskip
We first consider Sylow $p_i$-subgroups of a  regular subgroup of $\Aut(\Ga)$ isomorphic to the generalized quaternion group $\Q_{4n}$, where $\Ga=\Cay(\Q_{4n},S)$ is a normal Cayley graph.

\medskip
\begin{lemma}\rm\label{Hp_i1-s}
Let $\Ga=\Cay(\Q_{4n},S)$ be a normal Cayley graph of the generalized quaternion group $\Q_{4n}$ ($n\geq 2$) and let $H$ be a regular subgroup of $\Aut(\Ga)$ isomorphic to $\Q_{4n}$. Let $1\le i\le s$. Then $H_{p_i}=\la R(a_i)\s_{a_i}^{r_i}\ra$ for some integer $r_i$.
\end{lemma}

\proof Since $1\leq i\leq s$, $p_i$ is an odd prime and $n
\geq 3$. Set $$B_i=\la \s_{a_i}\ra \rtimes \Aut(C_{p_i^{k_i}})\mbox{ and } S_i=\la \s_{a_i}\ra \rtimes \Aut(C_{p_i^{k_i}})_{p_i}.$$
Then $S_i$ is the unique Sylow $p_i$-subgroup of $B_i$. By Eq~(\ref{holQ_4n2}),
$B_i\unlhd \Aut(\Q_{4n})$ and hence $S_i\unlhd \Aut(\Q_{4n})$.
Since $\la \s_{a_i}\ra\unlhd S_i$ and $\Aut(C_{p_i^{k_i}})$ is abelian, Proposition~\ref{Derived} implies that $S_i'\leq \la \s_{a_i}\ra$. By Proposition~\ref{pgroup}, $S_i$ is a regular $p_i$-group. Let $A=\Aut(\Ga)$. Then $\Aut(\Q_{4n},S)\leq A$. First we have the following claim:

\medskip
\noindent {\bf Claim:} $S_i\cap \Aut(\Q_{4n},S)\leq \la \s_{a_i}\ra$.

\medskip

	Let $1\not=z\in S_i\cap \Aut(\Q_{4n},S)$. Then $z=xy$ with $x\in \la \s_{a_i}\ra$ and $y\in \Aut(C_{p_i^{k_i}})_{p_i}$. To prove the claim, it suffices to show that $y=1$. We may assume that $x=\s_{a_i}^r$ and $a_i^y=a_i^s$ for some integers $r$ and $s$, where $(s,p_i)=1$. Let $o(y)=p_i^t$ for some non-negative integer $t$. Recall that $o(a_i)=p_i^{k_i}$ with $k_i\geq 1$. Then $a_i=a_i^{y^{p_i^t}}=a_i^{s^{p_i^t}}$, that is, $s^{p_i^t}=1\mod(p_i^{k_i})$. This implies that $s^{p_i^t}=1\mod (p_i)$, and by the well-known Fermat's Little Theorem, $s=1\mod (p_i)$. Since  $x^y=(\s_{a_i}^r)^y=\s_{a_i^y}^r=\s_{a_i^s}^r=(\s_{a_i}^r)^s=x^s$, we have $[x,y]\in \la x^{p_i}\ra$ and hence $\la x,y\ra'\leq \la x^{p_i}\ra$. Since $S_i$ is a regular $p_i$-group, $(xy)^{p_i}=x^{p_i}y^{p_i}z$ with $z\in \la x^{p_i^2}\ra$.
Let $o(x)=p_i^\ell$ for some non-negative integer $\ell$.

Suppose $o(x)=p_i^\ell\leq o(y)=p_i^t$, that is, $\ell \leq t$.  Since $(xy)^{p_i}=x^{p_i}y^{p_i}z$ with $z\in \la x^{p_i^2}\ra$, we have  $(xy)^{p_i^{\ell-1}}=x^{p_i^{\ell-1}}y^{p_i^{\ell-1}}$ and $(xy)^{p_i^{\ell}}=y^{p_i^{\ell}}$. Since $z=xy\in \Aut(\Q_{4n},S)$, we have $x^{p_i^{\ell-1}}y^{p_i^{\ell-1}}\in \Aut(\Q_{4n},S)$ and $y^{p_i^{\ell}}\in \Aut(\Q_{4n},S)$. If $\ell<t$ then $y^{p_i^{\ell}}\in \Aut(\Q_{4n},S)$ implies $\a_i\in  \Aut(\Q_{4n},S)$ as $y\in \Aut(C_{p_i^{k_i}})_{p_i}$. By Lemma~\ref{lm1}~(1), $\Ga$ is non-normal, contradicting the hypothesis. Thus we may let $\ell=t$. Then $\la x^{p_i^{\ell-1}}\ra=\la \b_i\ra$ and $\la y^{p_i^{\ell-1}}\ra=\la \a_i\ra$. From $x^{p_i^{\ell-1}}y^{p_i^{\ell-1}}\in \Aut(\Q_{4n},S)$ and $\la \b_i\ra\unlhd \Aut(\Q_{4n})$, we obtain $\Aut(\Q_{4n},S)\cap \la \b_i\ra\a_i\not=1$. Again by Lemma~\ref{lm1}~(1), $\Ga$ is non-normal, a contradiction.

Thus, $o(x)=p_i^\ell> o(y)=p_i^t$. Then $\la x^{p_i^{\ell-1}}\ra=\la \b_i\ra$ and $\la y^{p_i^{t-1}}\ra=\la \a_i\ra$. Write $d=a_i^{p_i^{k_i-1}}$. 

Suppose $t\geq 1$. Since $(xy)^{p_i}=x^{p_i}y^{p_i}z$ with $z\in \la x^{p_i^2}\ra$, we have  $(xy)^{p_i^{\ell-1}}=x^{p_i^{\ell-1}}\in \Aut(\Q_{4n},S)$ and $(xy)^{p_i^{t-1}}=x^{p_i^{t-1}}y^{p_i^{t-1}}z'\in \Aut(\Q_{4n},S)$ with $z'\in \la x^{p_i}\ra$. Thus,  $\b_i\in \Aut(\Q_{4n},S)$ and $\la \s_{a_i}\ra \a_i\cap \Aut(\Q_{4n},S)\not=\emptyset$. Note that $d$ has order $p_i$, $\b_i=\s_d$ and $\la \s_{a_i}\ra \a_i$ fixes every coset of $\la d\ra$ in $C_{2n}$.  By Lemma~\ref{lm1}~(2), $\Ga$ is non-normal, a contradiction. 

It follows that $t=0$ and $y=1$. Thus, $z=x\in \la \s_{a_i}\ra$, as claimed. \qed

Recall that $1\leq i\leq s$.
Since $H \cong R(\Q_{4n})\unlhd A$, Proposition~\ref{normal} and Eq~(\ref{holQ_4n1}) imply that
$H\le R(\Q_{4n})\rtimes\Aut(\Q_{4n},S)\le \Hol(\Q_{4n})=R(\Q_{4n})\rtimes \Aut(\Q_{4n})$. Since $p_i$ is an odd prime, $\Aut(C_{p_i^{k_i}})\cong C_{p_i^{k_i-1}(p_i-1)}$, and since $p_{s+1}=2$, $\Aut(C_{p_{s+1}^{k_{s+1}}})$ is a $2$-group. Recall that $p_1>p_2>\cdots>p_s>p_{s+1}$. Then $p_i\nmid  |\la \s_{a_j}\ra \rtimes \Aut(C_{p_j^{k_j}})|$ for every $i<j\leq s+1$. By Eq~(\ref{holQ_4n2}),
$\Aut(\Q_{4n})_{p_i}\leq (\la \s_{a_1}\ra \rtimes \Aut(C_{p_1^{k_1}}))\times \cdots \times (\la \s_{a_i}\ra \rtimes \Aut(C_{p_i^{k_i}}))$ for any Sylow $p_i$-subgroup $\Aut(\Q_{4n})_{p_i}$ of $\Aut(\Q_{4n})$.
Since $H_{p_i}R(\Q_{4n})/R(\Q_{4n})\leq \Hol(\Q_{4n})/R(\Q_{4n})\cong \Aut(\Q_{4n})$ and $H_{p_i}R(\Q_{4n})/R(\Q_{4n})$ is a $p_i$-subgroup, we have
\begin{equation*}\label{Hpi}
H_{p_i}\leq R(\Q_{4n})\Aut(\Q_{4n})_{p_i}\leq R(\Q_{4n})((\la \s_{a_1}\ra \rtimes \Aut(C_{p_1^{k_1}}))\times \cdots \times (\la \s_{a_i}\ra \rtimes \Aut(C_{p_i^{k_i}}))).
\end{equation*}

Take $h_i\in H_{p_i}$. Then $h_i$ is a $p_i$-element and we may write
\begin{equation}\label{h_i}
    h_i=R(z)\g_1\cdots\g_i, \mbox{\ where } z\in \Q_{4n} \mbox{\ and\ } \g_j\in B_j=\la \s_{a_j}\ra \rtimes \Aut(C_{p_j^{k_j}}) \mbox{ for } 1\leq j\leq i.
\end{equation}

Let $\g=\g_1\cdots\g_i$. Since $h_i\in A$ and $R(z)\in A$, we have $\g\in A$, and since $\g\in \Aut(\Q_{4n})$, we have $\g\in A_1=\Aut(\Q_{4n},S)$. Since $R(\Q_{4n})\unlhd \Hol(\Q_{4n})$, we have $1=h_i^{o(h_i)}\in R(\Q_{4n})\g^{o(h_i)}$, forcing $\g^{o(h_i)}=1$.
By Eq~(\ref{holQ_4n2}), $\g$ and $\g_j$ are $p_i$-elements for every $1\leq j\leq i$.

Since $R(\Q_{4n})\unlhd A$ and $n\geq 3$, $R(C_{2n})$ is normal in $A$ and hence the orbits $\la a\ra$ and $\la a\ra b$ of $R(C_{2n})$ consist of a complete imprimitive block system of $A$ on $\Q_{4n}$. Since $p_i$ is odd, $H_{p_i}$ fixes both $\la a\ra$ and $\la a\ra b$ setwise. Since $\g\in A_1$, $\g$ fixes
both $\la a\ra$ and $\la a\ra b$ setwise, and hence $R(z)$ fixes both $\la a\ra$ and $\la a\ra b$ setwise, forcing that $z\in C_{2n}$.

\medskip
Now we finish the proof by induction on $i$.

Let $i=1$. By Eq~(\ref{h_i}), $\g=\g_1\in B_1=\la \s_{a_1}\ra\rtimes\Aut(C_{p_1^{k_1}})$, and since $\g$ is a $p_1$-element, we have $\g\in S_1=\la \s_{a_1}\ra\rtimes\Aut(C_{p_1^{k_1}})_{p_1}$, because $\la \s_{a_1}\ra\unlhd B_1$. Since $\g\in \Aut(\Q_{4n},S)$, Claim implies that $\g\in \langle \s_{a_1}\ra$ and hence $h_1\in R(C_{2n})\times \langle \s_{a_1}\ra$, that is, $H_{p_1}\leq R(C_{2n})\times \langle \s_{a_1}\ra$. It follows that $H_{p_1}\leq \la R(a_1)\ra\times \langle \s_{a_1}\ra\cong \mz_{p_1^{k_1}}\times\mz_{p_1^{k_1}}$, because $\la R(a_1)\ra$ is the unique Sylow $p_1$-subgroup of $R(C_{2n})$. Since $H\cong R(\Q_{4n})$,  $H_{p_1}$ is a cyclic group of order $p_1^{k_1}$, and we may let $H_{p_1}=\la u\ra$ and $u=vw$ with $v\in \la R(a_1)\ra$ and $w\in \la \s_{a_1}\ra$, where $o(u)=p_1^{k_1}$. If $o(v)<p_1^{k_1}$ then $1\not=u^{o(v)}=(vw)^{o(v)}=w^{o(v)}\in \la \s_{a_1}\ra$, contradicting the semiregularity of $H_{p_1}$. Thus, $o(v)=p_1^{k_1}$, and hence $\la v\ra=\la R(a_1)\ra$. Then we may let $H_{p_1}=\la u\ra=\la R(a_1)\s_{a_1}^{r_1}\ra$ for some integer $r_1$. The lemma is true for $i=1$.

Let $i\geq 2$ and $1\leq j<i\leq s$. By inductive hypothesis, we assume that $H_{p_j}=\la R(a_j)\s_{a_j}^{r_j}\ra$ for some integer $r_j$ for each $1\leq j<i$. To finish the proof, we only need to show that $H_{p_i}=\la R(a_i)\s_{a_i}^{r_i}\ra$ for some integer $r_i$.

For any $h_i\in H_{p_i}$, by Eq~(\ref{h_i}) we have $h_i=R(z)\g_1\cdots\g_i$, where $z\in C_{2n}$ and $\g_j\in B_j=\la \s_{a_j}\ra \rtimes \Aut(C_{p_j^{k_j}})$. Recall $H_{p_j}=\la R(a_j)\s_{a_j}^{r_j}\ra$. Since $H\cong \Q_{4n}$, we have $[H_{p_j},H_{p_i}]=1$, and hence $[R(a_j)\s_{a_j}^{r_j},h_i]=1$. By Eq~(\ref{s_ai}), $[R(a_j)\s_{a_j}^{r_j},R(z)]=1$ as $z\in C_{2n}$. Furthermore, for any $1\leq \ell\leq i$ with $\ell\not=j$, Eq~(\ref{holQ_4n2}) implies that $[R(a_j)\s_{a_j}^{r_j},\g_\ell]=1$, where $\g_\ell\in B_\ell=\la \s_{a_\ell}\ra \rtimes \Aut(C_{p_\ell^{k_\ell}})$. It follows from Eq~(\ref{h_i}) that $[R(a_j)\s_{a_j}^{r_j},\g_j]=1$, that is, $(R(a_j)\s_{a_j}^{r_j})^{\g_j}=R(a_j)\s_{a_j}^{r_j}$. Let $\g_j=\d_1\d_2$ with $\d_1\in \la \s_{a_j}\ra$ and $\d_2\in \Aut(C_{p_j^{k_j}})$. Then $R(a_j)\s_{a_j}^{r_j}=(R(a_j)\s_{a_j}^{r_j})^{\d_1\d_2}=(R(a_j)\s_{a_j}^{r_j})^{\d_2}=
R(a_j^{\d_2})(\s_{a_j}^{r_j})^{\d_2}$. It follows that $R(a_j)=R(a_j^{\d_2})$ and so $a_j=a_j^{\d_2}$. Since $\d_2\in \Aut(C_{p_j^{k_j}})$, we have $\d_2=1$ and $\g_j=\d_1\in \la \s_{a_j}\ra$, and since $\g_j$ is an $p_i$-element and $j\not=i$, we have $1=\d_1=\g_j$.
Thus, $h_i=R(z)\g_i$, where $z\in C_{2n}$ and $\g_i\in B_i=\la \s_{a_i}\ra \rtimes \Aut(C_{p_i^{k_i}})$. Since $\g_i$ is a $p_i$-element, we have $\g_i\in S_i$ and $\g_i\in A_1=\Aut(\Q_{4n},S)$. By
Claim, $\g_i\in \la \s_{a_i}\ra$, and then the argument in the paragraph for $i=1$, by replacing $1$ with $i$, implies that $H_{p_i}=\la R(a_i)\s_{a_i}^{r_i}\ra$ for some integer $r_i$, as required. This completes the proof.
\qed

\medskip

\noindent{\bf Proof of Theorem~\ref{mainth1}:} Let $\Ga=\Cay(\Q_{4n},S)$ be a normal Cayley graph of the generalized quaternion group $\Q_{4n}$ ($n\geq 2$) and let $A=\Aut(\Ga)$. Then $S=S^{-1}$ and $R(\Q_{4n})\unlhd A=R(\Q_{4n})\rtimes \Aut(\Q_{4n},S)\leq \Hol(\Q_{4n})$. Let $n=2$ or $n$ is odd. By Proposition~\ref{ndci}, $\Q_{4n}$ is an NDCI-group, and hence a NCI-group. Thus, we may assume that $n\geq 4$ is even. Thus, $R(C_{2n})\unlhd \Hol(\Q_{4n})$, and by Eq~(\ref{factorizationof2n}), $k_{s+1}\geq 2$. 

Let $H\le A$ be a regular subgroup isomorphic to $\Q_{4n}$. Write
\begin{equation*}
    H=\la x,y\mid x^{2n}=1, y^2=x^n, x^y=x^{-1}\ra.
\end{equation*}
By Babai~\cite{Babai}, $\Ga$ is CI if and only
if all regular subgroups of $\Aut(\Ga)$ isomorphic to $\Q_{4n}$ are conjugate in $A$, and since $R(\Q_{4n})\unlhd A$,  $\Ga$ is CI if and only if $R(\Q_{4n})$ is the unique regular subgroup of $A$ isomorphic to $\Q_{4n}$. To finish the proof, it suffices to show that $H=R(\Q_{4n})$. 

By Lemma~\ref{Hp_i1-s}, if $1\leq j\leq s$ then $H_{p_j}=\la R(a_j)\s_{a_j}^{r_j}\ra$ for some integer $r_j$, and since $p_j$ is odd, we have $H_{p_j}\leq \langle x\rangle$. Set
$$B_i=\la \s_{a_i}\ra \rtimes \Aut(C_{p_i^{k_i}}), \mbox{ for each } 1\leq i\leq s+1.$$

Take a $2$-element $h$ in $H$. By Eqs~(\ref{holQ_4n1}) and (\ref{holQ_4n2}), we may write
\begin{equation}\label{h}
    h=R(z)\g=R(z)\g_1\cdots\g_{s+1}, \mbox{\ where } z\in \Q_{4n} \mbox{\ and\ } \g_i\in B_i \mbox{ for } 1\leq i\leq s+1.
\end{equation}
Then $\g=\g_1\cdots\g_{s+1}$, $\g\in \Aut(\Q_{4n},S)$ and $h^2=R(zz^{\g^{-1}})\g^2$. Since $R(\Q_{4n})\unlhd \Hol(\Q_{4n})$, we have $1=h^{o(h)}\in R(\Q_{4n})\g^{o(h)}=R(\Q_{4n})\g_1^{o(h)}\cdots\g_{s+1}^{o(h)}$, implying $1=\g^{o(h)}=\g_1^{o(h)}=\cdots=\g_{s+1}^{o(h)}$. Since $h$ is a $2$-element, $\g$ and $\g_i$ are $2$-elements. Recall that $C_{2n}=\la a\ra$ and $\Q_{4n}=C_{2n}\cup C_{2n}b$. Since $R(C_{2n})\unlhd \Hol(\Q_{4n})$, every subgroup of $\la R(a)\ra$ is characteristic in $R(\Q_{4n})$, and hence normal in $\Hol(\Q_{4n})$, which implies that the set of orbits of any subgroup of $R(C_{2n})$ on $\Q_{4n}$ forms a complete imprimitive block system of $\Aut(\Ga)$ and $\Hol(\Q_{4n})$. In particular, $\{C_{2n}, C_{2n}b\}$ is a complete imprimitive block system. From Eq~(\ref{h}) we know that if $h$ fixes $C_{2n}$ and $C_{2n}b$ setwise then $z\in C_{2n}$, and if $h$ interchanges $C_{2n}$ and $C_{2n}b$ then $z\in C_{2n}b$.

\medskip
Assume that $h$ fixes $C_{2n}$ and $C_{2n}b$ setwise. By Eq~(\ref{h}), we have
\begin{equation}\label{hfixesboth}
  h=R(z)\g=R(z)\g_1\cdots\g_{s+1}, \mbox{\ where } z\in C_{2n} \mbox{\ and\ } \g_i\in B_i \mbox{ for } 1\leq i\leq s+1.
\end{equation}

\medskip
\noindent{\bf Claim 1:} Assume $h$ fixes $C_{2n}$ and $C_{2n}b$ setwise. If $h\in \la x\ra$ then $z\in \la a_{s+1}\ra$ and $\g_j=1$ for all $1\leq j\leq s$. Furthermore, either $\g_{s+1}\in \la\s_{a_{s+1}}\ra$, or $k_{s+1}\geq 3$ and $o(h)=2$ or $4$. If $h\in \la x\ra y$, then $\g_j\in \la \s_{a_j}\ra\e_j$ for $1\leq j\leq s$ and one of the following occurs: (i) $\g_{s+1}\in \la \s_{a^n}\ra$ and $o(z_{s+1})=4$, and (ii) $k_{s+1}\geq 3$, $\g_{s+1}\in \la \s_{a^n}\ra \a_{s+1}$ and $o(z_{s+1})=4$.

\medskip
Let $h\in \la x\ra$ and $1\leq j\leq s$. Since $H_{p_j}=\la R(a_j)\s_{a_j}^{r_j}\ra\leq \la x\ra$, we have $[R(a_j)\s_{a_j}^{r_j},h]=1$, and since $z\in C_{2n}$, we have  $[R(a_j)\s_{a_j}^{r_j}, R(z)]=1$. For any $1\leq \ell\leq s+1$ with $\ell\not=j$, we have $[\g_\ell,R(a_j)\s_{a_j}^{r_j}]=1$ as $\g_\ell\in B_\ell=\la \s_{a_\ell}\ra \rtimes \Aut(C_{p_\ell^{k_\ell}})$. By Eq~(\ref{hfixesboth}), $[\g_j,R(a_j)\s_{a_j}^{r_j}]=1$. It follows  $R(a_j)\s_{a_j}^{r_j}=(R(a_j)\s_{a_j}^{r_j})^{\g_j}=R(a_j^{\g_j})\s_{a_j^{\g_j}}^{r_j}$. Thus, $R(a_j^{\g_j})=R(a_j)$, that is, $a_j^{\g_j}=a_j$. Since $B_j=\la \s_{a_j}\ra \rtimes \Aut(C_{p_j^{k_j}})$, we have $\g_j\in \la \s_{a_j}\ra$, and hence $\g_j=1$ as $\g_j$ is a $2$-element.

Thus, $h=R(z)\g_{s+1}$, and $\g_{s+1}\in \Aut(\Q_{4n},S)$. Since $z\in C_{2n}$, we may write $z=z_1z_2$ with $z_1\in \la a_1a_2\cdots a_s\ra$ and $z_2\in \la a_{s+1}\ra$. Then  $h=R(z_1)R(z_2)\g_{s+1}$ and  $[R(z_1),R(z_2)\g_{s+1}]=1$. Since $o(R(z_1))$ is odd and $o(R(z_2)\g_{s+1})$ is a power of $2$, we have $R(z_1)=1$ as $h$ is a $2$-element. It follows $z\in \la a_{s+1}\ra$. Let $\g_{s+1}=\d_1\d_2\in \la \s_{a_{s+1}}\ra\rtimes \Aut(C_{2^{k_{s+1}}})$ for some $\d_1\in \la \s_{a_{s+1}}\ra$ and $\d_2\in \Aut(C_{2^{k_{s+1}}})$.

Suppose $o(\d_2)\geq 4$. Then $k_{s+1}\geq 4$ and $\Aut(C_{2^{k_{s+1}}})\cong \mz_2\times\mz_{2^{k_{s+1}-2}}$. Clearly, $\Aut(C_{2^{k_{s+1}}})$ has three involutions, that is, $\e_{s+1}$, $\a_{s+1}$ and $\e_{s+1}\a_{s+1}$, and $\d_2^{o(\d_2)/2}=\a_{s+1}$. Note that $h^{o(\d_2)/2}\in \la R(a_{s+1})\ra\g_{s+1}^{o(\d_2)/2}\subseteq \la R(a_{s+1})\ra\la \s_{a_{s+1}}\ra\d_2^{o(\d_2)/2}=\la R(a_{s+1})\ra\la \s_{a_{s+1}}\ra\a_{s+1}$. It follows that  $\la\s_{a_{s+1}}\ra\a_{s+1}\cap \Aut(\Q_{4n},S)\not=1$. By Lemma~\ref{lm1}~(2), $\Ga$ is non-normal, a contradiction.

It follows that $o(\d_2)\leq 2$ and $\d_2=1,\a_{s+1},\e_{s+1}$ or $\a_{s+1}\e_{s+1}$. Since $\g_{s+1}\in \Aut(\Q_{4n},S)$, Lemma~\ref{lm1}~(2) implies $\d_2\not=\a_{s+1}$. If $\d_2=1$ then $\g_{s+1}\in \la \s_{a_{s+1}}\ra$, as required. We may assume that $k_{s+1}\geq 2$ and $\d_2=\e_{s+1}$ or $\a_{s+1}\e_{s+1}$. If $k_{s+1}=2$ then $\d_2=\e_{s+1}=\a_{s+1}$ or $\d_2=\a_{s+1}\e_{s+1}=1$, which have been considered above. Thus, $k_{s+1}\geq 3$, and $h\in \la R(a_{s+1})\ra \la \s_{a_{s+1}}\ra\e_{s+1}$ or $\la R(a_{s+1})\ra \la \s_{a_{s+1}}\ra\e_{s+1}\a_{s+1}$. By Eqs~(\ref{auorderp_i}) and (\ref{involutions}), every element of  $\la R(a_{s+1})\ra \la \s_{a_{s+1}}\ra\e_{s+1}$ has order $2$ and every element of $\la R(a_{s+1})\ra \la \s_{a_{s+1}}\ra\e_{s+1}\a_{s+1}$ has order $2$ or $4$. Thus, $o(h)=2$ or $4$.

Now let $h\in \la x\ra y$. Then $o(h)=4$ and for each $1\leq j\leq s$, we have $(R(a_j)\s_{a_j^{r_j}})^h=R(a_j^{-1})\s_{a_j^{-r_j}}$ as $H_{p_j}=\la R(a_j)\s_{a_j}^{r_j}\ra\leq \la x\ra$.  It follows that $R(a_j^{-1})\s_{a_j^{-r_j}}=(R(a_j)\s_{a_j^{r_j}})^h=(R(a_j)\s_{a_j^{r_j}})^{\g_j}=
R(a_j^{\g_j})\s_{(a_j^{r_j})^{\g_j}}$. This implies that $a_j^{\g_j}=a_j^{-1}$, and since $\g_j\in B_j=\la \s_{a_j}\ra \rtimes \Aut(C_{p_j^{k_j}})$, we have $\g_j\in \la \s_{a_j}\ra\e_j$. In particular,  $o(\g_j)=2$. 

Similarly let $\g_{s+1}=\d_1\d_2$ with $\d_1\in\la \s_{a_{s+1}}\ra$ and $\d_2\in \Aut(C_{2^{k_{s+1}}})$. 
By the preceding paragraph beginning with ``Suppose $o(\d_2)\geq 4$.'', we have $\d_2^2=1$ and $\d_2=1,\a_{s+1},\e_{s+1}$ or $\a_{s+1}\e_{s+1}$. Since $o(\g_j)=2$ for all $1\leq j\leq s$, we have $h^2=R(z)\g R(z)\g=R(zz^{\g^{-1}})\g_{s+1}^2$. It follows that $zz^{\g^{-1}}\in C_{2n}$ and  $zz^{\g^{-1}}\not=1$ because $\la h^2\ra $ is semiregular. Note that $\g_{s+1}^2\in \la \s_{a_{s+1}}\ra \d_2^2= \la \s_{a_{s+1}}\ra$. Since $[\la \s_{a_{s+1}}\ra, R(C_{2n})]=1$ and $o(h^2)=2$, we have $h^2=R(a^n)$ or $R(a^n)\s_{a^n}$, and since $R(a^n)\s_{a^n}$ fixes $b$, we have $h^2=R(a^n)$ and $\g_{s+1}^2=1$. In particular, $zz^{\g^{-1}}=a^n$. Recall that $z\in C_{2n}$. If $\d_2=1$, then $\g_{s+1}^2=1$ implies $\g_{s+1}\in \la \s_{a^n}\ra$, and $zz^{\g^{-1}}=a^n$ implies $o(z_{s+1})=4$, which is the (i) in the claim. Thus, we may let $\d_2\not=1$, implying $k_{s+1}\geq 2$. If $\d_2=\e_{s+1}$, then $zz^{\g^{-1}}=zz^\e=1$, a contradiction. Now we may let $\d_2=\a_{s+1}$ or $\e_{s+1}\a_{s+1}$. For $k_{s+1}=2$, we have $\d_2=\a_{s+1}=\e_{s+1}$ or $\d_2=\e_{s+1}\a_{s+1}=1$, which has been considered above. Thus, we may assume that $k_{s+1}\geq 3$. 
If $\d_2=\e_{s+1}\a_{s+1}$, then $\g=\g_1\g_2\cdots \g_{s+1}\in \la \s_a\ra\e\a_{s+1}$. By Lemma~\ref{lm1}~(3), $\Ga$ is non-normal as $\g\in \Aut(\Q_{4n},S)$, a contradiction. Thus,  $\d_2=\a_{s+1}$. From $\g_{s+1}^2=1$ and $zz^{\g^{-1}}=a^n$, we have $\g_{s+1}\in \la \s_{a^n}\ra \a_{s+1}$ and $o(z_{s+1})=4$. This completes the proof of Claim~1.
\qed


\medskip
Assume that $h$ interchanges $C_{2n}$ and $C_{2n}b$. By Eq~(\ref{h}) we have
\begin{equation}\label{hinterchangesboth}
h=R(zb)\g=R(zb)\g_1\cdots\g_{s+1}, \mbox{ where } z\in C_{2n} \mbox{ and } \g_i\in B_i \mbox{ for } 1\leq i\leq s+1.
\end{equation}

\medskip
\noindent{\bf Claim 2:} Assume $h$ interchanges $C_{2n}$ and $C_{2n}b$. Then $h^2\in \la R(a_{s+1})\ra\times \la \s_{a_{s+1}}\ra$. Let $1\leq j\leq s$. If $h\in \la x\ra$ then either $\g_j=1$ with $H_{p_j}=\la R(a_j) \s_{a_j}^2\ra$, or $\g_j=\s_{z_j^{-2}}\e_j\in \la \s_{a_j}\ra \e_j$ with $H_{p_j}=\la R(a_j)\ra$, and $\g_{s+1}\in \la \s_{a_{s+1}}\ra\d$ with $\d\in \Aut(C_{2^{k_{s+1}}})$ and  $\d^2=1$. If $h\in \la x\ra y$ then either $\g_j=1$ with $H_{p_j}=\la R(a_j)\ra$, or $\g_j=\s_{z_j^{-2}}\e_j\in \la \s_{a_j}\ra \e_j$ with $H_{p_j}=\la R(a_j) \s_{a_j}^2\ra$, and one of the following occurs: (i) $\g_{s+1}=1$, (ii) $k_{s+1}\geq 3$, $\g_{s+1}=\s_{z_{s+1}^{-2}}\e_{s+1}$, and (iii) $k_{s+1}\geq 3$, $\g_{s+1}=\s_{z_{s+1}^{-2-2^{k_{s+1}-1}}}\e_{s+1}\a_{s+1}$, in which (i) cannot occur if $H_{p_j}=\la R(a_j) \s_{a_j}^2\ra$ for some $1\leq j\leq s$.

\medskip

Let $1\leq j\leq s$. Recall that $H_{p_j}=\la R(a_j)\s_{a_j}^{r_j}\ra\leq \la x\ra$ for some integer $r_j$. By Eq~(\ref{s^R}), $$(\s_{a_j}^{r_j})^{R(zb)}=R(a_j^{r_j})\s_{a_j}^{r_j} \mbox{ and } (R(a_j)\s_{a_j}^{r_j})^{R(zb)}=R(a_j^{r_j-1})\s_{a_j}^{r_j}.$$

Since $\g_i\in B_i=\la \s_{a_i}\ra \rtimes \Aut(C_{p_i^{k_i}})$, we may write  $\g_i=\d_{1_i}\d_{2_i}$ with $\d_{1_i}\in \la \s_{a_i}\ra$ and $\d_{2_i}\in \Aut(C_{p_i^{k_i}})$, for all $1\leq i\leq s+1$. Note that $[\g_i,R(a_j^{r_j-1})\s_{a_j}^{r_j}]=1$ for $i\not=j$.

Assume $h\in \la x\ra$. Then $[R(a_j)\s_{a_j}^{r_j},h]=1$. Thus,  $R(a_j)\s_{a_j}^{r_j}=(R(a_j)\s_{a_j}^{r_j})^h=(R(a_j)\s_{a_j}^{r_j})^{R(zb)\g}=(R(a_j^{r_j-1})\s_{a_j}^{r_j})^{\g_j}$.
It follows that $(a_j^{r_j})^{\g_j}=a_j^{r_j}$ and $(a_j^{r_j-1})^{\g_j}=a_j$. Since $\d_{1_j}\in \la \s_{a_j}\ra$, we have $(a_j^{r_j})^{\d_{2_j}}=a_j^{r_j}$ and $(a_j^{r_j-1})^{\d_{2_j}}=a_j$. It follows that $a_j^{\d_{2_j}}=a_j^{r_j-1}$ and $a_j^{r_j}=(a_j^{r_j})^{\d_{2_j}}=a_j^{r_j(r_j-1)}$. This implies that $r_j(r_j-2)=0\mod (p_j^{k_j})$. Note that $p_j$ is an odd prime. If $(r_j,p_j)=1$ then $r_j=2\mod (p_j^{k_j})$ and $a_j^{\d_{2_j}}=a_j^{r_j-1}=a_j$, that is, $\d_{2_j}=1$ and $\g_j=1$  because $\g_j=\d_{1_j}\in \la \s_{a_j}\ra$ is a $2$-element, implying $H_{p_j}=\la R(a_j)\s_{a_j}^2\ra$. If $(r_j,p_j)\not=1$ then $(r_j-2,p_j)=1$ and $r_j=0\mod (p_j^{k_j})$, implying $\d_{2_j}=\e_j$, $H_{p_j}=\la R(a_j)\ra$, and $\g_j\in \la \s_{a_j}\ra\e_j$. It follows that $\g_j^2=1$ for all $1\leq j\leq s$ and $\g^2=\g_{s+1}^2$. 

Recall that $h^2=R(zb(zb)^{\g^{-1}})\g^2=R(zb(zb)^{\g^{-1}})\g_{s+1}^2$. Since $h^2$ is a $2$-element, we have $zb(zb)^{\g^{-1}} =z(z^{-1})^{\g^{-1}}bb^{\g^{-1}}\in \la a_{s+1}\ra$. Note that  $\g_{s+1}^2\in \la \s_{a_{s+1}}\ra\d_{2_{s+1}}^2$ and $\g_{s+1}^2\in \Aut(\Q_{4n},S)$ as $\g\in \Aut(\Q_{4n},S)$. If $\d_{2_{s+1}}^2\not=1$, then $\la \s_{a_{s+1}}\ra \a_{s+1}\cap \Aut(\Q_{4n},S)\not=\emptyset$, and by Lemma~\ref{lm1}~(2), $\Ga$ is non-normal, a contradiction. Thus, $\d_{2_{s+1}}^2=1$, and  $h^2=R(z(z^{-1})^{\g^{-1}}bb^{\g^{-1}})\g^2\in \la R(a_{s+1})\ra \g_{s+1}^2\subseteq \la R(a_{s+1})\ra\times \la \s_{a_{s+1}}\ra$. For $\g_j\in \la \s_{a_j}\ra\e_j$, write $\g_j=\s_{a_j^{s_j}}\e_j$, and the $j$-part of $z(z^{-1})^{\g^{-1}}bb^{\g^{-1}}$ is $a_j^{s_j}z_j^{2}$ as $z\in C_{2n}$, where $z_j$ is the $j$-part of $z$, and since $z(z^{-1})^{\g^{-1}}bb^{\g^{-1}}\in \la a_{s+1}\ra$, we have $a_j^{s_j}=z_j^{-2}$, that is, $\g_j=\s_{z_j^{-2}}\e_j$. 

Now assume $h\in \la x\ra y$. Then $(R(a_j)\s_{a_j}^{r_j})^h=R(a_j^{-1})\s_{a_j}^{-r_j}$. Thus, $R(a_j^{-1})\s_{a_j}^{-r_j}=(R(a_j)\s_{a_j}^{r_j})^{R(zb)\g}=(R(a_j^{r_j-1})\s_{a_j}^{r_j})^{\g_j}$.
It follows that $(a_j^{r_j})^{\g_j}=a_j^{-r_j}$ and $(a_j^{r_j-1})^{\g_j}=a_j^{-1}$, implying $(a_j^{r_j})^{\d_{2_j}}=a_j^{-r_j}$ and $(a_j^{r_j-1})^{\d_{2_j}}=a_j^{-1}$. Then, $a_j^{\d_{2_j}}=a_j^{1-r_j}$ and $a_j^{-r_j}=(a_j^{r_j})^{\d_{2_j}}=a_j^{r_j(1-r_j)}$. This means that  $r_j(r_j-2)=0\mod (p_j^{k_j})$. If $(r_j,p_j)=1$ then $r_j=2\mod (p_j^{k_j})$ and $a_j^{\d_{2_j}}=a_j^{1-r_j}=a_j^{-1}$, that is, $\d_{2_j}=\e_j$ and $\g_j\in \la \s_{a_j}\ra\e_j$, implying  $H_{p_j}=\la R(a_j)\s_{a_j}^2\ra$ and $\g_j^2=1$. If $(r_j,p_j)\not=1$ then $r_j=0\mod (p_j^{k_j})$ and $a_j^{\d_{2_j}}=a_j$, implying  hence $\g_j=1$ and $H_{p_j}=\la R(a_j)\ra$. Thus, $\g_j^2=1$ for all $1\leq j\leq s$. By the argument in the above paragraph, $\d_{2_{s+1}}^2=1$, $h^2=R(z(z^{-1})^{\g^{-1}}bb^{\g^{-1}})\g_{s+1}^2\in \la R(a_{s+1})\ra\times\la \s_{s+1}\ra$, and if $\g_j\in \la \s_{a_j}\ra\e_j$ then $\g_j=\s_{z_j^{-2}}\e_j$. 

Note that $o(h)=4$. Since $\la h^2\ra $ is semiregular,$h^2=R(a^n)$, implying that $z(z^{-1})^{\g^{-1}}bb^{\g^{-1}}=a^n$ and $\g_{s+1}^2=1$. Recall that $\g_{s+1}=\d_{1_{s+1}}\d_{2_{s+1}}\in B_{s+1}$ for $\d_{1_{s+1}}\in \la \s_{a_{s+1}}\ra$ and $\d_{2_{s+1}}\in \Aut(C_{2^{k_{s+1}})})$. Then $\d_{2_{s+1}}^2=1$ implies $\d_{2_{s+1}}=1,\e_{s+1},\a_{s+1}$ or $\e_{s+1}\a_{s+1}$.  

Recall that $\g_j=1$ or $\s_{z_j^{-2}}\e_j$, and $\g_j\not=1$ if and only if $H_{p_j}=\la R(a_j)\s_{a_j}^2\ra$, in which $\s_{a_j}\in \Aut(\Q_{4n},S)$  as $p_j$ is odd. Set $\pi=\{p_k\ |\ \g_k\not=1, 1\leq k\leq s\}$ and $a_{\pi}=\prod_{p_k\in \pi}a_k$. Then $\s_{a_\pi}\in \Aut(\Q_{4n},S)$, and since $\s_{a^n}\in \Aut(\Q_{4n},S)$, we have $\s_{a^na_\pi}\in \Aut(\Q_{4n},S)$. 

Suppose $\d_{2_{s+1}}=\a_{s+1}$. It is easy to see that $\g$ fixes every coset of $\la a^na_\pi\ra$ in $C_{2n}$. By Lemma~\ref{lm1}, $\Ga$ is non-normal, a contradiction. Thus, $\d_{2_{s+1}}\not=\a_{s+1}$. 

Applying $\g_{s+1}^2=1$ and $z(z^{-1})^{\g^{-1}}bb^{\g^{-1}}=a^n$, we have that if $\d_{2_{s+1}}=1$ then $\g_{s+1}=1$, which is (i) of Claim~2. Thus, we assume 
$\d_{2_{s+1}}=\e_{s+1}$ or $\a_{s+1}\e_{s+1}$. Recall that  $k_{s+1}\geq 2$. For $k_{s+1}=2$, we have $\d_{2_{s+1}}=1$ or $\a_{s+1}$, which has been considered above. Thus, we may assume that $k_{s+1}\geq 3$. Again applying $\g_{s+1}^2=1$ and $z(z^{-1})^{\g^{-1}}bb^{\g^{-1}}=a^n$ we have that if $\d_{2_{s+1}}=\e_{s+1}$ then $\g_{s+1}=\s_{z_{s+1}^{-2}}\e_{s+1}$, and if $\d_{2_{s+1}}=\a_{s+1}\e_{s+1}$ then $\g_{s+1}=\s_{z_{s+1}^{-2-2^{k_{s+1}-1}}}\e_{s+1}\a_{s+1}$.

Assume that there is some $1\leq j\leq s$ such that $\g_j\not=1$. Then $1\not=\s_{a_\pi}\in \Aut(\Q_{4n},S)$. Suppose that $\d_{2_{s+1}}=1$. Then $\g$ fixes every coset of $\la a_\pi\ra$ in $C_{2n}$. By Lemma~\ref{lm1}, $\Ga$ is non-normal, a contradiction. Thus, $\d_{2_{s+1}}\not=1$, that is, (i) cannot occur. \qed

Recall that $H$ is regular on $\Q_{4n}$. Note that either $x$ fixes $C_{2n}$ and $C_{2n}b$, or interchanges $C_{2n}$ and $C_{2n}b$. For the former, $y$ interchanges $C_{2n}$ and $C_{2n}b$. For the latter, if $y$ fixes $C_{2n}$ and $C_{2n}b$, then $xy$ interchanges $C_{2n}$ and $C_{2n}b$, and then we replace $y$ by $xy$ because $x$ and $xy$ have the same relations as $x$ and $y$ in $\Q_{4n}$. Thus, we always assume that $y$ interchanges $C_{2n}$ and $C_{2n}b$. By Eq~(\ref{h}), we may write 
\begin{equation}\label{y}
   y=R(db)\lambda=R(db)\lambda_1\cdots\lambda_{s+1}, \mbox{ where } d\in C_{2n} \mbox{ and } \lambda_i\in B_i \mbox{ for each } 1\leq i\leq s+1.
\end{equation}
By Claim~2, we have the following observation.

\medskip
\noindent{\bf Observation 1:} Either $\lambda_j=1$ with $H_{p_j}=\la R(a_j)\ra$, or $\lambda_j=\s_{d_j^{-2}}\e_j\in \la \s_{a_j}\ra \e_j$ with $H_{p_j}=\la R(a_j) \s_{a_j}^2\ra$, and one of the following occurs: (i) $\lambda_{s+1}=1$, (ii) $k_{s+1}\geq 3$, $\lambda_{s+1}=\s_{d_{s+1}^{-2}}\e_{s+1}$, and (iii) $k_{s+1}\geq 3$, $\lambda_{s+1}=\s_{d_{s+1}^{-2-2^{k_{s+1}-1}}}\e_{s+1}\a_{s+1}$. Furthermore, if there is some $1\leq j\leq s$ such that $H_{p_j}=\la R(a_j) \s_{a_j}^2\ra$, then (i) cannot occur, forcing $k_{s+1}\geq 3$. Note that $\lambda^2=1$.

\medskip 
Since $o(x)=o(a)=2n$, we have the following unique decomposition:
$$x=x_1x_2\cdots x_sx_{s+1}, \mbox{ where } o(x_i)=p_i^{k_i} \mbox{ for } 1\leq i\leq s+1.$$
Let $1\leq j\leq s$. Since $H_{p_j}$ is the unique Sylow $p_j$-subgroup of $H$, we have $H_{p_j}=\la x_{j} \ra$, and so $H_{p_j}$ and $x_j$ fix $C_{2n}$ and $C_{2n}b$. Thus, $x$ interchanges $C_{2n}$ and $C_{2n}b$ if and only if $x_{s+1}$ interchanges $C_{2n}$ and $C_{2n}b$. We divide the proof into two cases depending on whether $x$ fixes $C_{2n}$ and $C_{2n}b$ or not.

\medskip
\noindent{ \bf Case I:} $x$ interchanges $C_{2n}$ and $C_{2n}b$.

\medskip
We will derive contradictions, thereby concluding that this case cannot occur.

Since $x_j$ fixes  $C_{2n}$ and $C_{2n}b$ for all $1\leq j\leq s$, $x_{s+1}$ interchanges $C_{2n}$ and $C_{2n}b$. By Eq~(\ref{h}), we may write 
\begin{equation}\label{x-s+1}
 x_{s+1}=R(cb)\mu=R(cb)\mu_1\cdots\mu_{s+1}, \mbox{ where } c\in C_{2n} \mbox{ and } \mu_i\in B_i \mbox{ for } 1\leq i\leq s+1.
\end{equation}
By Claim~2, we have the following observation.

\medskip
\noindent{\bf Observation 2:} $\mu_{s+1}=\s_{a_{s+1}^r}\d\in \la \s_{a_{s+1}}\ra\d$ with $\d^2=1$ and $\d\in \Aut(C_{2^{k_{s+1}}})$, and for $1\leq j\leq s$, either $\mu_j=1$ with $H_{p_j}=\la R(a_j) \s_{a_j}^2\ra$, or $\mu_j=\s_{c_j^{-2}} \e_j\in \la \s_{a_j}\ra \e_j$ with $H_{p_j}=\la R(a_j)\ra$. 

\medskip
Note that $o(x_{s+1})=2^{k_{s+1}}$ and $x_{s+1}^2=R(cb(cb)^{\mu^{-1}})\mu^2=R(cb(cb)^{{\mu_{s+1}^{-1}}})\s_{a_{s+1}^r(a_{s+1}^r)^\d}$. Then 
\begin{equation}\label{s+1squre}
 x_{s+1}^2=R(c_{s+1}c_{s+1}^{-\d}a_{s+1}^{r}a^n)\s_{a_{s+1}^r(a_{s+1}^r)^\d}.
\end{equation}

Since $x^y=x^{-1}$, we have $x_{s+1}^{y}=x_{s+1}^{-1}$, and so $x_{s+1}^{-1}=\mu^{-1}R(cb)^{-1}=R(cb^{-1})^\mu\mu^{-1}$. It follows that $R(cb^{-1})^\mu\mu^{-1}=x_{s+1}^y=R(cb)^{R(db)\lambda}\mu^{R(db)\lambda}$ and hence
$$R(cb^{-1})^\mu\mu^{-1}=R(c^{-1}d^2b)^\lambda R(b^{-1}d^{-1}d^{\mu^{-1}}b^{\mu^{-1}})^\lambda\mu^\lambda=R(c^{-1}dd^{\mu^{-1}}b^{\mu^{-1}})^\lambda\mu^\lambda.$$
Thus, $\mu^\lambda=\mu^{-1}$ and so $\lambda\mu^{-1}=\mu\lambda$. Since $\lambda^2=1$, we have $(\lambda\mu)^2=1$. By Observations 1 and 2, we have $\lambda_j\mu_j\in \la \s_{a_j}\ra\e_j$ for all $1\leq j\leq s$, and hence $(\lambda_{s+1}\mu_{s+1})^2=1$. 
Furthermore, $R(c b^{-1})=R(c^{-1}dd^{\mu^{-1}} b^{\mu^{-1}})^{\mu\lambda}=R((c^{-1}d)^\mu d b)^{\lambda}$, that is, 
$c=(c^{-1}d)^{\mu\lambda} d^{\lambda}b^{\lambda}b$.
Considering the $s+1$-parts of this equation, we have 
\begin{equation}\label{s+1parts}
 c_{s+1}=(c_{s+1}^{-1}d_{s+1})^{\mu_{s+1}\lambda_{s+1}}d_{s+1}^{\lambda_{s+1}}b^{\lambda_{s+1}}b \mbox{ and } (\lambda_{s+1}\mu_{s+1})^2=1.
\end{equation}


\medskip
\noindent {\bf Claim~3:} $\lambda=1$, $\mu_{s+1}^2=1$, and $y=R(db)$. Furthermore, $\mu_j=\s_{c_j^{-2}} \e_j$  and $H_{p_j}=\la R(a_j)\ra$ for all $1\leq j\leq s$.  

Suppose $\lambda_{s+1}\not=1$. Recall that $\mu_{s+1}=\s_{a_{s+1}^r}\d$. 
By Eq~(\ref{y}) and Observation~1, we have $k_{s+1}\geq 3$ and either $\lambda_{s+1}=\s_{d_{s+1}^{-2}}\e_{s+1}$ or  $\lambda_{s+1}=\s_{d_{s+1}^{-2-2^{k_{s+1}-1}}}\e_{s+1}\a_{s+1}$. 

Note that $yx_{s+1}\in H$ fixes $C_{2n}$ and $C_{2n}b$, and $o(yx_{s+1})=4$. Recall that $\lambda^2=1$ and $(\lambda\mu)^2=1$. By Eqs~(\ref{y}) and (\ref{x-s+1}), we have  
\begin{equation}\label{yx-s+1}
yx_{s+1}=R(db(cb)^{\lambda})\lambda\mu, \mbox{ and } (yx_{s+1})^2=R(db(cb)^\lambda ((db)^\lambda cb)^\mu).
\end{equation}   

Let $\lambda_{s+1}=\s_{d_{s+1}^{-2}}\e_{s+1}$.  Since $\lambda_j\mu_j\in \la \s_{a_j}\ra\e_j$ for all $1\leq j\leq s$, we have $\lambda\mu\in \la\s_a\ra \e\delta$, and by Claim~1, $\e_{s+1}\delta=1$ or $\a_{s+1}$, that is, $\delta=\e_{s+1}$ or $\e_{s+1}\a_{s+1}$. Similarly, for $\lambda_{s+1}=\s_{d_{s+1}^{-2-2^{k_{s+1}-1}}}\e_{s+1}\a_{s+1}$ we have $\e_{s+1}\a_{s+1}\delta=1$ or $\a_{s+1}$, and also $\delta=\e_{s+1}$ or $\e_{s+1}\a_{s+1}$. Thus, we always have $\delta=\e_{s+1}$ or $\e_{s+1}\a_{s+1}$, and one of the following four cases occurs: 
\begin{itemize}
    \item[(i)] $(\lambda_{s+1},\mu_{s+1})=(\s_{d_{s+1}^{-2}}\e_{s+1},\s_{a_{s+1}^r}\e_{s+1})$;
    \item[(ii)] $(\lambda_{s+1},\mu_{s+1})=(\s_{d_{s+1}^{-2}}\e_{s+1},\s_{a_{s+1}^r}\e_{s+1}\a_{s+1})$;
  \item[(iii)] $(\lambda_{s+1},\mu_{s+1})=(\s_{d_{s+1}^{-2-2^{k_{s+1}-1}}}\e_{s+1}\a_{s+1},\s_{a_{s+1}^r}\e_{s+1})$;
  \item[(iv)] $(\lambda_{s+1},\mu_{s+1})=(\s_{d_{s+1}^{-2-2^{k_{s+1}-1}}}\e_{s+1}\a_{s+1},\s_{a_{s+1}^r}\e_{s+1}\a_{s+1})$. 
  \end{itemize}
Note that $w^{2^{k_{s+1}-1}}=1$ for any $w\in \la a_{s+1}^2\ra$. By Eq~(\ref{s+1parts}), $(c_{s+1}^{-1}d_{s+1})^4=1$  for cases (i)-(iv). From $(\lambda_{s+1}\mu_{s+1})^2=1$ we have  $(a_{s+1}^r)^2=d_{s+1}^{-4}$ or $d_{s+1}^{-4}(a_{s+1}^r)^{{2^{k_{s+1}-1}}}$ for all cases (i)-(iv).  Note that $k_{s+1}\geq 3$.
If $(a_{s+1}^r)^2=d_{s+1}^{-4}(a_{s+1}^r)^{{2^{k_{s+1}-1}}}$ then $(a_{s+1}^r)^{{2^{k_{s+1}-1}}}=(d_{s+1}^{-4}(a_{s+1}^r)^{{2^{k_{s+1}-1}}})^{{2^{k_{s+1}-2}}}=1$, forcing $a_{s+1}^r\in \la a_{s+1}^2\ra$, and hence $(a_{s+1}^r)^{{2^{k_{s+1}-1}}}=1$. Thus, we always have $(a_{s+1}^r)^2=d_{s+1}^{-4}$. Since $\delta=\e_{s+1}$ or $\e_{s+1}\a_{s+1}$, we have $(c_{s+1}c_{s+1}^{-\d}a_{s+1}^{r}a^n)^2=(c_{s+1}d_{s+1}^{-1})^4=1$, and by Eq~(\ref{s+1squre}), $x_{s+1}^4=1$, contradicting that $o(x_{s+1})=2^{k_{s+1}}\geq 8$. 

Thus, $\lambda_{s+1}=1$. Since  $(\lambda_{s+1}\mu_{s+1})^2=1$, we have $\mu_{s+1}^2=1$. 
By Eq~(\ref{y}) and Observation~1, we have $\lambda_j=1$ and $H_{p_j}=\la R(a_j)\ra$ for all $1\leq j\leq s$, implying $\lambda=1$ and $y=R(db)$. By Eq~(\ref{x-s+1}) and Observation~2, we have 
$\mu_j=\s_{c_j^{-2}} \e_j\in \la \s_{a_j}\ra \e_j$ and $H_{p_j}=\la R(a_j)\ra$ for all $1\leq j\leq s$. This completes the proof of Claim~3. \qed

Recall that $k_{s+1}\geq 2$. By Observation~2, $\mu_{s+1}=\s_{a_{s+1}^r}\d$ and $\d^2=1$. By Claim~3, $y=R(db)$, and by Eq~(\ref{yx-s+1}), $yx_{s+1}=R(dc^{-1}a^n)\mu$ fixes $C_{2n}$ and $C_{2n}b$ setwise. By Claim~1, $\d=1$ or $\d=\a_{s+1}$ with $k_{s+1}\geq 3$. 

Assume $k_{s+1}=2$. Then $\d=1$, and $\mu_{s+1}^2=1$ implies $a_{s+1}^r=1$ or $a^n$. If $a_{s+1}^r=a^n$, then by Eq~(\ref{s+1squre}), $x_{s+1}^2=R(c_{s+1}c_{s+1}^{-\d}a_{s+1}^ra^n)=1$, contrary to that $o(x_{s+1})=2^{k_{s+1}}=4$. Thus, $a_{s+1}^r=1$ and hence $\mu_{s+1}=1$. Recall that $n\geq 4$ is even. Since $k_{s+1}=2$, we have $s\geq 1$ and $\mu_{s+1}=1=\e_{s+1}\a_{s+1}$. By Claim~3, $1\not=\mu\in \la \s_a\ra \e\a_{s+1}$. By Eq~(\ref{x-s+1}), $\mu\in\Aut(\Q_{4n},S)$, and by Lemma~\ref{lm1}~(3), $\Ga$ is non-normal, a contradiction.

Finally assume $k_{s+1}\geq 3$. Then $\d=1$ or $\a_{s+1}$. By Claim~3, $\mu_{s+1}^2=1$, that is, ${a_{s+1}^r}({a_{s+1}^r})^\d=1$. For $\d=1$, we have $a_{s+1}^r=1$ or $a^n$. For $\d=\a_{s+1}$, we have $(a_{s+1}^r)^2(a_{s+1}^r)^{2^{k_{s+1}-1}}=1$ and hence $(a_{s+1}^r)^4=1$, which implies $(a_{s+1}^r)^{2^{k_{s+1}-1}}=1$ as $k_{s+1}\geq 3$. Thus, we always have $a_{s+1}^r=1$ or $a^n$. 
By Eq~(\ref{s+1squre}), $x_{s+1}^4=R((c_{s+1}c_{s+1}^{-\d}a_{s+1}^ra^n)^2)=1$, contrary to that $o(x_{s+1})=2^{k_{s+1}}\geq 8$.  \qed


\medskip
\noindent{ \bf Case II:} $x$ fixes $C_{2n}$ and $C_{2n}b$.

In this case, $x_{s+1}$ fixes $C_{2n}$ and $C_{2n}b$, and $o(x_{s+1})=2^{k_{s+1}}$. By Claim~1, we have  
$$x_{s+1}=R(c)\s_{a_{s+1}^t}, \mbox{ for some } c\in \la a_{s+1}\ra \mbox{ and some integer } t.$$
Note that $[R(c),\s_{a_{s+1}^t}]=1$. Since $o(x_{s+1})=2^{k_{s+1}}$, the semiregularity of $\la x_{s+1}\ra$ implies that  $x_{s+1}^{2^{k_{s+1}-1}}=R(a^n)$ or $R(a^n)\s_{a^n}$. The latter cannot happen because $R(a^n)\s_{a^n}$ fixes $b$. Thus, $x_{s+1}^{2^{k_{s+1}-1}}=R(a^n)$, and hence $c^{2^{k_{s+1}-1}}=a^n$ and $(a_{s+1}^t)^{2^{k_{s+1}-1}}=1$, that is, $o(c)=2^{k_{s+1}}$ and $o(a_{s+1}^t)\not=2^{k_{s+1}}$. 

By Eq~(\ref{s^R}), $(\s_{a_{s+1}^t})^{R(db)}=R(a_{s+1}^t)\s_{a_{s+1}^t}$. Note that $c\in \la a_{s+1}\ra$ and $y$ is given in Eq~(\ref{y}). So we have  
$R(c^{-1})\s_{a_{s+1}^{-t}}=x_{s+1}^{-1}=x_{s+1}^y=
R(c)^{R(db)\lambda}(\s_{a_{s+1}^t})^{R(db)\lambda}=
R(c^{-1}a_{s+1}^t)^{\lambda_{s+1}}\s_{(a_{s+1}^t)^{\lambda_{s+1}}}$. Thus, 
\begin{equation}\label{fixingC2n}
(c^{-1}a_{s+1}^t)^{\lambda_{s+1}}=c^{-1} \mbox{ and }  (a_{s+1}^t)^{\lambda_{s+1}}=a_{s+1}^{-t}.
\end{equation}

\medskip
\noindent {\bf Claim~4:} $\lambda=1$, $y=R(db)$ and $H_{p_j}=\la R(a_j)\ra$ for all $1\leq j\leq s$.  

Suppose $\lambda_{s+1}\not=1$. By Eq~(\ref{y}) and Observation~1,  we have $k_{s+1}\geq 3$ and either $\lambda_{s+1}=\s_{d_{s+1}^{-2}}\e_{s+1}$ or  $\lambda_{s+1}=\s_{d_{s+1}^{-2-2^{k_{s+1}-1}}}\e_{s+1}\a_{s+1}$. For every $1\leq j\leq s$, we have 
either $\lambda_j=1$ with $H_{p_j}=\la R(a_j)\ra$, or $\lambda_j=\s_{d_j^{-2}}\e_j\in \la \s_{a_j}\ra \e_j$ with $H_{p_j}=\la R(a_j) \s_{a_j}^2\ra$.  
Set $\pi=\{p_k\ |\ \lambda_k\not=1, 1\leq k\leq s\}$ and $a_{\pi}=\prod_{p_k\in \pi}a_k$. Then $\s_{a_\pi}\in \Aut(\Q_{4n},S)$.

Since $\lambda_{s+1}=\s_{d_{s+1}^{-2}}\e_{s+1}$ or  $\s_{d_{s+1}^{-2-2^{k_{s+1}-1}}}\e_{s+1}\a_{s+1}$, Eq~(\ref{fixingC2n}) implies that $a_{s+1}^t=c^2$ or $c^2a^n$. Since $k_{s+1}\geq 3$, we have $o(a_{s+1}^t)=2^{k_{s+1}-1}$, that is, $\la a_{s+1}^t\ra=\la a_{s+1}^2\ra$. Since $x_{s+1}=R(c)\s_{a_{s+1}^t}$, we have $\s_{a_{s+1}^t}\in \Aut(\Q_{4n},S)$, implying $\s_{a_{s+1}^2}\in \Aut(\Q_{4n},S)$, and hence $\s_{a_\pi a_{s+1}^2}\in \Aut(\Q_{4n},S)$. It is easy to see that $\lambda_{s+1}$ fixes every coset of $\la a_{s+1}^2\ra$ in $C_{2n}$, and so $\lambda$ fixes every coset of $\la a_\pi a_{s+1}^2\ra$ in $C_{2n}$. By Lemma~\ref{lm1}~(2), $\Ga$ is non-normal, a contradiction. 

Thus, $\lambda_{s+1}=1$, and by Observation~1, $H_{p_j}=\la R(a_j)\ra$ and $\lambda_j=1$  for all $1\leq j\leq s$. It follows that $\lambda=1$ and $y=R(db)$, as required. \qed   

By Claim~4, $\lambda_{s+1}=1$, and by Eq~(\ref{fixingC2n}), $a_{s+1}^t=1$, that is, $x_{s+1}=R(c)$. It follows that $H=\la R(a_1),R(a_2),\cdots, R(a_s),x_{s+1},y\ra=R(\Q_{4n})$, as required. This completes the proof.
\qed

\section{Proof of Theorem~\ref{mainth3}}

In this section, we prove Theorem~\ref{mainth3}, and use the notions and formulae in Eqs~(\ref{factorizationof2n})-(\ref{s_ai}) of Section~\ref{S3}. We first construct an infinite family of NNN Cayley digraphs of $\Q_{4n}$.

\begin{lemma}\rm\label{NNND}
  Let $n\ge 6$ be even. Then the Cayley digraph $\Ga=\Cay(\Q_{4n},S)$ is an NNN digraph with $S=\{a,a^{n-1},b,a^{n/2}b,a^nb,a^{3n/2}b\}$.
\end{lemma}

\proof  Since $n$ is even, $(n-1,2n)=1$ and hence $o(a^{n-1})=2n$. This implies $$\a:\ a\mapsto a^{n-1}, b\mapsto b$$
induces an automorphism of $\Q_{4n}$. Clearly, $o(\a)=2$ and $\a\in \Aut(\Q_{4n},S)$. 

Let $\b=\s_{a^{n/2}}$ be given in Eq~(\ref{s_ai}), an automorphism of $\Q_{4n}$ induced by 
$$\b:\ a\mapsto a, b\mapsto a^{n/2}b.$$
Clearly, $\b$ cyclically permutes the four elements in $\{b,a^{n/2}b,a^nb,a^{3n/2}b\}$. Thus, $o(\b)=4$ and $\b\in\Aut(\Q_{4n},S)$. It is easy to check that if $n/2$ is even then  $\a\b\a=\b^{-1}$ and  $\la\a,\b\ra\cong D_8$, and if $n/2$ is odd then $\a\b=\b\a$ and $\la\a,\b\ra\cong \mz_4\times \mz_2$. 

\medskip
\noindent {\bf Claim}: $\Ga$ is normal and $\Aut(\Ga)=R(\Q_{4n})\rtimes \la\a,\b\ra$. 

Write $A=\Aut(\Ga)$. Note that $\la\a,\b\ra\leq \Aut(\Q_{4n},S)\leq A_1$ and $|\la\a,\b\ra|=8$. By Proposition~\ref{normal}, to finish the proof of the claim it suffices to show that $|A_1|=8$.  

Let  $S=S_1\cup S_2$ with $S_1=\{a,a^{n-1}\}$ and $S_2=\{b,a^{n/2}b,a^nb,a^{3n/2}b\}$. Note that $S_1\cap S_1^{-1}=\emptyset$ and $S_2=S_2^{-1}$. Set $\Ga_1=\Cay(\Q_{4n},S_1)$ and $\Ga_2=\Cay(\Q_{4n},S_2)$, and write $B=\Aut(\Ga_1)$ and $C=\Aut(\Ga_2)$. Then $\Ga=\Ga_1\cup \Ga_2$. An arc $(x,y)$ in $\Ga$ is called {\em an edge} if $(y,x)$ is also an arc in $\Ga$, and an arc $(x,y)$ in $\Ga$ is called a {\em directed edge} if $(y,x)$ is not an arc in $\Ga$. It is easy to see that $A$ fixes the set of edges of $\Ga$ setwise, and the set of directed edges of $\Ga$ setwise. Clearly, all directed edges of $\Ga$ are the same as the directed edges in $\Ga_1$, and all edges of $\Ga$ are the same as the edges in $\Ga_2$. It follows that $A\leq \Aut(\Ga_1)$ and $A\leq \Aut(\Ga_2)$, that is, $A\leq B\cap C$. For a subset $T$ of $\Q_{4n}$, denote by $A_{(T)}$ the subgroup of $A$ fixing $T$ pointwise in $\Ga$, and for $g\in \Q_{4n}$, denote by $A_g^*$ the subgroup of $A_g$ fixing the out-neighbourhood of $g$ in $\Ga$ pointwise. Similarly, we can define $B_{(T)}$ and $B_g^*$ for $\Ga_1$, and $C_{(T)}$ and $C_g^*$ for $\Ga_2$. 

Note that $\Ga_1$ has two connected components, that is, the induced subgraphs  $[C_{2n}]_{\Ga_1}$ and $[C_{2n}b]_{\Ga_1}$. Clearly, $[C_{2n}]_{\Ga_1}=\Cay(C_{2n},S_1)$ and $R(C_{2n})$ is transitive on both $[C_{2n}]_{\Ga_1}$ and $[C_{2n}b]_{\Ga_1}$. For $g\in \Q_{4n}$, denote by $\Ga_1(g^*)$ the set of out-neighbourhoods $g$ in $\Ga_1$. It is easy to see that $\Ga_1(1^*)=\{a,a^{n-1}\}$, and $a$ and $a^{n-1}$ have exactly one common out-neighbour $a^n$. This implies that $B_1^*$ fixes $\{a^2,a^n,a^{-2}\}$ pointwise, which is the $2$-distance out-neighbourhood of $1$.  It follows that    
$B_1^*=B_g^*$ for any $g\in \Ga_1(1^*)$, and by conjugating this equation by elements in $R(C_{2n})$, the connectedness of $[C_{2n}]_{\Ga_1}$ implies that $B_1^*$ fixes $C_{2n}$ pointwise. By conjugate of $R(b)$ on $B_1^*$, we have that $B_b^*$ fixes $C_{2n}b$ pointwise. Since $\Ga_1$ has out-valency $2$, $B_{(1,a)}=B_1^*$, and $|B_1|=2|B_1^*|$ as $\a\in B_1$. Note that $\a\in B_1$ has non-trivial action on both $C_{2n}$ and $C_{2n}b$. Since $B_1^*$ fixes $C_{2n}$ pointwise, the restriction $B_1^{C_{2n}}$ of $B_1$ on $C_{2n}$ is the restriction $\la\a\ra^{C_{2n}}$, that is, $B_1^{C_{2n}}=\la\a\ra^{C_{2n}}$, because $|B_1^{C_{2n}}|=2|B_{(1,a)}^{C_{2n}}|=2|(B_1^*)^{C_{2n}}|=2$. Similarly, $B_b^{C_{2n}b}=\la\a\ra^{C_{2n}b}$.
Since $[C_{2n}]_{\Ga_1}$ and $[C_{2n}b]_{\Ga_1}$ are components of $\Ga_1$, $C_{2n}$ and $C_{2n}b$ form a complete block system of $B$ and $A$. It follows that $B_1^{C_{2n}}=\la\a\ra^{C_{2n}}\leq \Aut([C_{2n}]_{\Ga_1})_1$ and $B_b^{C_{2n}b}=\la\a\ra^{C_{2n}b}\leq \Aut([C_{2n}b]_{\Ga_1})_b$.

Clearly, $A_1$ has two orbits on $S$, that is, $S_1$ and $S_2$. Note that $\la \a\ra$ is transitive on $S_1$ and $\la \b\ra$ is transitive on $S_2$. Since $|S_1|=2$ and $|S_2|=4$, we have $|A_1|=8|A_{(1,a,b)}|$. To finish the proof of the claim, we only need to show that $A_{(1,a,b)}=1$.   

Since $B_{(1,a)}=B_1^*$, we have $A_{(1,a,b)}\leq B_{(1,a,b)}\leq B_1^*$, and hence $A_{(1,a,b)}$ fixes $C_{2n}$ pointwise. Let $M=\la a^{n/2}\ra$. Then $M$ is a subgroup of $C_{2n}$ of order $4$, and so normal in $\Q_{4n}$. Since $S_2=\{b,a^{n/2}b,a^nb,a^{3n/2}b\}=Mb$, $\Ga_2$ is the union of $n/2$ copies of $K_{4,4}$, that is, $\Ga_2=\cup_{i=0}^{n/2-1}[Ma^i\cup Mba^i]_{\Ga_2}$, where $[Ma^i\cup Mba^i]_{\Ga_2}\cong K_{4,4}$ with partite sets $Ma^i$ and $Mba^i$. Since $A_{(1,a,b)}$ fixes $C_{2n}$ pointwise, it fixes the set $[C_{2n}b:M]$ of right cosets of $M$ in $C_{2n}b$ pointwise. In particular,  $A_{(1,a,b)}$ fixes $Mba^{1-n}$ and $Mba^{-1}$. Note that $o(a^{n-2})=2n/(n-2,2n)=n\geq 6$. Then $a^{n-2}\not\in M$, and hence $Mba^{1-n}\not=Mba^{-1}$. On the other hand, $A_{(1,a,b)}\leq B_b$ and hence $A_{(1,a,b)}$ fixes $\{ab,a^{n-1}b\}$ setwise, because it is the set of the out-neighbors of $b$ in $\Ga_1$. Thus, $A_{(1,a,b)}$ fixes $Mba^{-1}\cap \{ab,a^{n-1}b\}=\{ab\}$, and similarly $a^{n-1}b$. It follows that $A_{(1,a,b)}\leq B_b^*$, implying that $A_{(1,a,b)}$ fixes $C_{2n}b$ pointwise. Thus,  $A_{(1,a,b)}=1$. This completes the proof of Claim. \qed    

By Claim, $\Ga$ is normal. Write
$$H=\la R(ab)\a,R(b) \ra.$$
To finish the proof, it suffices to show that $H$ is regular, $H\cong \Q_{4n}$, and $H\ntrianglelefteq A$.  

\medskip
Since $(R(ab))^{\a}=R((ab)^{\a})=R(a^{n-1}b)$, we have 
\begin{align*}
  (R(ab)\a)^2=R(ab)R((ab)^{\a})=R(aba^{n-1}b)=R(a^2).
\end{align*}
Since $n$ is even, we have that $o(R(ab)\a)$ is even. It follows that $$n=o(R(a^2))=o(R(ab)\a)/(o(R(ab)\a),2)=o(R(ab)\a)/2,$$ that is, $o(R(ab)\a)=2n$.
Moreover,
\begin{align*}
  (R(ab)\a)^{R(b)}=R(b^{-1})R(ab)\a R(b)=R(b^{-1}ab^2)\a=\a R(ba^{n-1})=(R(ab)\a)^{-1}.
\end{align*}
Thus, $\la R(ab)\a,R(b)\ra \cong R(\Q_{4n})$.

Since $\Ga$ is normal, $R(a^2)$ is normal in $A$, and hence its orbit set $[\Q_{4n},\la a^2\ra]$ is a complete block system of $A$. Clearly, $R(b)$ interchanges $\la a^2\ra$ and $\la a^2\ra b$, and $\la a^2\ra a$ and $\la a^2\ra ab$. It is easy to see that $R(ab)\a$ interchanges $\la a^2\ra$ and $\la a^2\ra ab$. Thus, $H$ is transitive on  $[\Q_{4n},\la a^2\ra]$. Note that $(R(ab)\a)^2=R(a^2)$ and $\la R(a^2)\ra$ is transitive on every coset of $\la a^2\ra$ in $\Q_{4n}$. It follows that $H$ is transitive on $\Q_{4n}$. Since $|H|=4n$, $H$ is regular on $\Q_{4n}$. 

Suppose $H$ is normal in $A$. Then $H^\b=H$. Recall that $\b=\s_{a^{n/2}}$. Then $R(b)^\b=R(a^{n/2}b)\in H$, and $R(a^{n/2}b)R(b)=R(a^{3n/2})\in H$. 
 
Let $n/2$ be odd. Since $(R(ab)\a)^2=R(a^2)\in H$ and $(2,3n/2)=1$, we have $R(a)\in H$ as $R(a^{3n/2})\in H$. Since $R(b)\in H$, we have $H=R(\Q_{4n})$, which is impossible. 

Let $n/2$ be even. Recall that $\a\b\a=\b^{-1}$. Since $R(b)^{-1}R(ab)\a=R(a^{-1})\a\in H$, we have $R(a)\a\in H$. Then $(R(a)\a)^\b=R(a)\a^\b\in H$. Furthermore, $(R(a)\a)^{-1}R(a)\a^\b=\a\a^\b\in H$. It follows that $\a\a^\b=\b^2\in H$, contradicting the regularity of $H$. 

Thus, $H\ntrianglelefteq A$, as required. This completes the proof of Lemma~\ref{NNND}. \qed

\medskip

\noindent{\bf Proof of Theorem~\ref{mainth3}:} Let $\Q_{4n}$ be the generalized quaternion group of order $4n$ with $n\geq 2$. Let $n\ge 6$ be even. By Lemma~\ref{NNND}, $\Q_{4n}$ is an NNND-group. On the other hand, let $\Q_{4n}$ be an NNND-group. If $n=2$ or $n$ is odd, by Proposition~\ref{ndci} we have that $\Q_{4n}$ is an NDCI-group, implying that $\Q_{4n}$ is not an NNND-group. Thus, $n\not=2$ and $n$ is even, that is, $n\geq 4$ is even. For $n=4$, using Magma~\cite{Magma} we may check that $\Q_{16}$ is not an NNND-group. Thus, $n\geq 6$ is even. \qed

\end{document}